\documentclass{amsart}

\usepackage{amsmath,amsthm,amssymb,array,enumerate,parskip,graphicx,etoolbox,tabularx}
\usepackage[all]{xy}

\usepackage{xcolor} 

\usepackage[pagebackref]{hyperref}
\definecolor{dark-red}{rgb}{0.5,0.15,0.15}
\definecolor{dark-blue}{rgb}{0.15,0.15,0.6}
\definecolor{dark-green}{rgb}{0.15,0.6,0.15}
\hypersetup{
    colorlinks, linkcolor=dark-red,
    citecolor=dark-blue, urlcolor=dark-green
}

\usepackage[capitalise]{cleveref}
\usepackage{comment}

\newcommand{\Z}{\mathbb{Z}} 
\newcommand{\Q}{\mathbb{Q}} 
\newcommand{\F}{\mathbb{F}} 
\newcommand{\Fp}{\mathbb{F}_p} 
\newcommand{\Fpbar}{\bar{\mathbb{F}}_p} 

\newcommand \Hom {\mathop{\rm Hom}}

\newcommand{\Gal}[1]{G_{#1}} 
\newcommand{\rGal}[2]{\mathrm{Gal}(#1/#2)} 

\newcommand{\e}{\epsilon} 

\newcommand{\dsh}{\bar d^{\prime}} 

\newcommand \BA {\mathbb{A}}
\newcommand \PP {{\mathbb P}^1}
\newcommand \ZZ {{\mathbb Z}}
\newcommand \CC {{\mathbb C}}
\newcommand \QQ {{\mathbb Q}}
\newcommand \NN {{\mathbb N}}
\newcommand  \FF {{\mathbb F}}
\newcommand \Jac {\mathop {\rm Jac}}
\newcommand \Ker {\mathop {\rm Ker}}
\newcommand \Tr {\mathop {\rm Tr}}
\newcommand \GW {\mathop {\rm GW}}

\DeclareMathOperator{\dlog}{dlog}

\DeclareMathOperator{\Ext}{Ext}

\newtheorem{theorem}{Theorem}[section]

\newtheorem{lemma}[theorem]{Lemma} 
\newtheorem{corollary}[theorem]{Corollary}
\newtheorem{example}[theorem]{Example}
\newtheorem{proposition}[theorem]{Proposition} \theoremstyle{definition}
\newtheorem{remark}[theorem]{Remark} 
\newtheorem{definition}[theorem]{Definition}

\newtheorem{question}[theorem]{Question}

\newcommand{\Aa}{\Z/p} 

\title[Galois action and cohomology of Fermat curves]{The Galois action and cohomology of a relative homology group of Fermat curves}
\author{Rachel Davis}
\address{University of Wisconsin-Madison}
\email{rachel.davis@wisc.edu}
\author{Rachel Pries}
\address{Colorado State University}
\email{pries@math.colostate.edu}
\author{Vesna Stojanoska}
\address{University of Illinois at Urbana-Champaign}
\email{vesna@illinois.edu}
\author{Kirsten Wickelgren}
\address{Georgia Institute of Technology}
\email{kwickelgren3@math.gatech.edu}

\thanks{We would like to thank BIRS for hosting the WIN3 conference where we began this project
and AIM for support for this project through a Square collaboration grant. 
We would like to thank a referee for helpful comments.
Some of this work was done while the third and fourth authors were in
residence at MSRI during the spring 2014 Algebraic topology semester, supported by NSF grant 0932078 000.
The second author was supported by NSF grant DMS-15-02227. The third author was supported by NSF grants DMS-1606479 and DMS-1307390. The fourth author was supported by an American Institute of Mathematics five year fellowship and NSF grants DMS-1406380 and DMS-1552730. }

\begin{document}

\begin{abstract}
For an odd prime $p$ satisfying Vandiver's conjecture, we give explicit formulae for the 
action of the absolute Galois group $G_{\QQ(\zeta_p)}$ on the homology of the degree $p$ Fermat curve, building on work of Anderson.
Further, we study the invariants and the first Galois cohomology group which are associated with obstructions to rational points on the Fermat curve.
 \end{abstract}

\maketitle


\section{Introduction}

In this paper, we study the action of the absolute Galois group on the homology of the Fermat curve.
Let $p$ be an odd prime, 
let $\zeta$ be a chosen primitive $p$th root of unity, and consider the cyclotomic field $K={\mathbb Q}(\zeta)$.
Let $\Gal{K}$ be the absolute Galois group of $K$.
The Fermat curve of exponent $p$ is the smooth projective curve 
$X \subset {\mathbb P}_K^2$ of genus $g=(p-1)(p-2)/2$ given by the equation
\[x^p + y^p = z^p.\] 

Anderson \cite{Anderson} proved several foundational results about the Galois module structure of a certain relative homology 
group of the Fermat curve. These results are closely related to \cite{Ihara} \cite{Coleman89}, and were further developed in \cite{AI-prol} \cite{Anderson_hyper}.
Consider the affine open $U \subset X$ given by $z \not = 0$, which has equation $x^p+y^p=1$.
Consider the closed subscheme $Y \subset U$ defined by $xy=0$, which consists of $2p$ points. 
Let $H_1(U,Y; \Z/p)$ denote the \'etale homology group, 
with $\Z/p$ coefficients, of the pair $(U \otimes \overline{K},Y \otimes \overline{K})$;
it is a continuous module over $\Gal{{\mathbb Q}}$. There is a $\mu_p \times \mu_p$ action on $X$ given by 
\[(\zeta^i, \zeta^j) \cdot [x,y,z] = [\zeta^i x,\zeta^j y,z] , ~~~~~~(\zeta^i, \zeta^j) \in \mu_p \times \mu_p,\] 
which determines an action on $U$, preserving $Y$.
By \cite[Theorem 6]{Anderson}, the group $H_1(U,Y; \Z/p)$ is a free rank one $\ZZ/p[\mu_p \times \mu_p]$ module, with generator denoted $\beta$.    
The Galois action of $\sigma \in \Gal{K}$ is then determined by $\sigma \beta = B_{\sigma} \beta$, for some unit $B_\sigma \in  \ZZ/p[\mu_p \times \mu_p]$.
 
Let $L$ be the splitting field of $1-(1-x^p)^p$.
By \cite[Section 10.5]{Anderson}, the $\Gal{K}$ action on $H_1(U,Y; \Z/p)$ 
factors through $\rGal{L}{K}$.  
This implies that the full $\Gal{K}$ module structure of $H_1(U,Y; \Z/p)$ 
is determined by the finitely many elements $B_q$ for $q \in \rGal{L}{K}$.

From Anderson's work, the description of the elements $B_q$ is theoretically complete in the following sense:
 Anderson shows that $B_{q}$ is determined by an analogue of the classical gamma function 
$\Gamma_{q} \in \overline{\FF}_p[\mu_p] \simeq \overline{\FF}_p[\epsilon]/\langle \epsilon^p-1 \rangle$.
By \cite[Theorems 7 \& 9]{Anderson}, 
there is a formula $B_q = \dsh(\Gamma_{q})$ (with $\dsh$ as defined in \cref{sec:determineAction}).
The canonical derivation $d: \overline{\FF}_p[\mu_p] \to \Omega \overline{\FF}_p[\mu_p]$ 
to the module of K\"ahler differentials allows one to take the logarithmic derivative $\dlog \Gamma_{q}$ of $\Gamma_{q}$. 
Since $p$ is prime, $\dlog \Gamma_{q}$ determines $B_{q}$ uniquely \cite[10.5.2, 10.5.3]{Anderson}. The function $q \mapsto \dlog \Gamma_{q}$ is in turn determined by a relative homology group of the punctured affine line $H_1(\mathbb{A}^1 - V(\sum_{i=0}^{p-1} x^i), \{ 0, 1\}; \ZZ/p)$ \cite[Theorem 10]{Anderson}. 

In this paper, for any odd prime $p$ satisfying Vandiver's conjecture, 
we extend Anderson's work for the Fermat curve of exponent $p$ 
by finding a closed form formula for $B_q$ for all $q \in \rGal{L}{K}$.
This formula is valuable for calculating Galois cohomology of the Fermat curve and other applications.

We now describe the results of the paper in more detail.
We assume throughout that $p$ is an odd prime satisfying Vandiver's Conjecture, namely that $p$ does not divide the order $h^+$ of the class group of $\QQ(\zeta+\zeta^{-1})$; 
this is true for all $p$ less than 163 million and all regular primes. 
Under this condition, we proved in \cite{WINF:birs} 
that $\rGal{L}{K} \simeq (\ZZ/p)^{r+1}$ where $r = (p-1)/2$.
More precisely, let $\kappa$ denote the classical Kummer map; i.e., for $\theta \in K^*$, let 
$\kappa(\theta): \Gal{K} \to \mu_p$ be defined by
\[\kappa (\theta) (\sigma) = \frac{\sigma \sqrt[p]{\theta}}{ \sqrt[p]{\theta}}.\] 
Then the map
\[\kappa(\zeta) \times \prod_{i = 1}^{\frac{p-1}{2}}\kappa(1-\zeta
^{-i}) : \rGal{L}{K} \to (\mu_p)^{\frac{p+1} {2}}\] 
is an isomorphism \cite[Corollary 3.7]{WINF:birs}, 
We review this material and give additional information about the extension $\rGal{L}{\Q}$
in \Cref{sec:Galois}. 

In \cref{sec:determineAction}, we review \cite[Corollary 4.2]{WINF:birs}
which gives a formula for $\dlog\Gamma_q$ in terms of the above description of $\rGal{L}{K}$,
see \eqref{eq:dlogGamma} and \eqref{eq:cis}. 
Writing $\dlog\Gamma_q= \sum_{i=1}^{p-1} c_i \e^i \dlog \e$ for $c_i$ in $\Fp$, 
we note that each $c_i$ is linear in the coordinate projections of $q$ viewed as an element of $(\Fp)^{\frac{p+1}{2}} \cong (\mu_p)^{\frac{p+1}{2}}$, which is isomorphic
to $\rGal{L}{K}$ since a $p$th root of unity has been chosen.

In \cref{Sexplicit}, we use this formula to compute a closed form formula for $B_{q}$
in terms of the generators $\e_0$ and $\e_1$ for $\Lambda_1=\Aa[\mu_p \times \mu_p]$. 
As the first step, in \Cref{Gamma_from_c}, 
we determine $\Gamma_q \in \overline{\FF}_p[\mu_p]$ from 
$\dlog\Gamma_q$ using a truncated exponential map $E_0$ defined in \eqref{E0} 
and an auxiliary polynomial $\gamma$ defined in \eqref{Egamma}. 
As the second step, we determine $B_q$ from $\Gamma_q$, 
and re-express the result in terms of a second exponential map $E_1$ defined in \eqref{E1}.
Although $\gamma$ has coefficients in $\Fpbar$, the resulting $B_q$ is indeed in $\Lambda_1$.
This yields the first main result.

 \begin{theorem} \label{Tintro} (see \Cref{cor:Beta})
Suppose $p$ is an odd prime satisfying Vandiver's conjecture. 
Then the action of $\rGal{L}{K}$ on the relative homology 
\[H_1(U, Y; \Z/p) \cong \Lambda_1=\Aa[\mu_p \times \mu_p]\] of the Fermat curve is determined as follows. For $q \in \rGal{L}{K}\cong (\Fp)^{\frac{p+1}{2}}$, let the image of $q$ in $(\Fp)^{\frac{p+1}{2}}$ be \[q=(c_0, c_1, \ldots, c_{\frac{p-1}{2}}) \] and for $i >\frac{p-1}{2}$, let $c_i  = c_{p-i} - i c_0,$ and $c = \sum_{i=1}^{p-1} c_i$. Let $F \in \Fpbar$ be a solution to the equation \[F^p - F + \sum_{i=1}^{p-1} c = 0.\]
Let \[ \gamma(\e) = \sum_{i=1}^{p-1} \left(\frac{c_i + c -F}{i}\right)\e^i - \sum_{i=1}^{p-1} \frac{c_i}{i}.\]
Then $q$ acts by multiplication by the element $B_q \in \Lambda_1$ with the explicit formula
\[B_q =\frac{E_0(\gamma(\e_0)) E_0(\gamma(\e_1))}{E_0(\gamma(\e_0 \e_1))} =  \frac{E_1( \gamma(\e_0) + \gamma(\e_1))}{ E_0(\gamma(\e_0\e_1) )  },\] where $E_0$ and $E_1$ are the truncated exponential maps of \eqref{E0} and \eqref{E1}, respectively.
\end{theorem}

\Cref{Snorm} contains two applications of \Cref{Tintro}, which hold for any odd prime $p$
satisfying Vandiver's conjecture.
By \cite[9.6 and 10.5.2]{Anderson}, if $q \in \rGal{L}{K}$, then
$B_q - 1$ lies in the augmentation ideal $(1-\epsilon_0)(1-\epsilon_1) \Lambda_1$;
this is equivalent to the statement that $(B_q-1) \beta \in H_1(U; \Aa)$.
In \Cref{cor:BqinI3}, we provide a technical strengthening of \cite[9.6 and 10.5.2]{Anderson}.
The first application, \Cref{Tnorm}, is that the norm of $B_q$ is $0$ or, equivalently, 
that the norm of $q$ acts as zero on $H_1(U, Y; \Z/p)$, for almost all $q \in \rGal{L}{K}$;
the only exception is when $p=3$ and $q$ does not fix $\zeta_9 \in L$. 
This result is of significant importance in computing Galois cohomology, as seen in \Cref{SGalCohomology}.
For the second application, note that Anderson's result implies that $H_1(U; \Aa)$ is trivialized by 
the product $\prod_{i=1}^{p-1} (B_{q_i} - 1)$ for any $q_1, \ldots, q_{p-1} \in Q$;
The improvement in \Cref{cor:BqinI3} allows us to show in \Cref{application2}
that in fact $H_1(U; \Aa)$ is trivialized by the product of only $s = \lfloor 2p/3 \rfloor$ such terms.

Having explicitly determined the action of $Q=\rGal{L}{K}$, and therefore of $\Gal{K}$, on $M=H_1(U,Y;\Z/p)$, we proceed to studying the zeroth and first associated Galois cohomology groups. 
In \Cref{Sinvariant}, we study the $\Gal{K}$-invariants, which are just the $Q$-invariants 
since the action of $\Gal{K}$ factors through $Q$.
In \Cref{Ph1umq}, we prove that ${\rm codim}(H_1(U; \ZZ/p)^Q, M^Q)=2$ for all odd $p$ and find a uniform subspace of $M^Q$ in \Cref{lem:InvSubspace,lem:otherinvariants}; we use these results in future work.  

In \Cref{SGalCohomology}, we work towards determining the first Galois cohomology group. 
Initially, the material in this section might seem disjoint from the earlier sections.
However, these general results in commutative algebra will eventually play a key role in 
understanding obstructions for rational points on Fermat curves.

For the second main result, consider an extension of finite (or profinite) groups 
\begin{equation}\label{Egaloisexactintro} 
1\to N \to G \to Q \to 1.
\end{equation}
Suppose $M$ is a $\ZZ[G]$-module on which $N$ acts trivially. Note that this applies to $G=\Gal{K}$, $Q=\rGal{L}{K}$, and $M=H_1(U,Y; \Z/p)$.
Consider the differential in the spectral sequence associated with \eqref{Egaloisexactintro}
\[d_2: H^1(N, M)^{Q} \to H^2(Q, M).\] 
It gives a short exact sequence
\[  0 \to H^1(Q,M ) \to H^1(G,M) \to \Ker d_2 \to 0, \]
which reduces the calculation of $H^1(G,M)$ to the two simpler calculations of $H^1(Q,M)$ and $\Ker d_2$. We address the first of these calculations in \Cref{SH1Q}, while the rest of \Cref{SGalCohomology} concerns the second.

When $Q \simeq (\Z/p)^{r+1}$, we determine the kernel of $d_2$ algebraically.
To state the result about ${\Ker}(d_2)$, fix a set of generators $\tau_0, \ldots, \tau_r$ of $Q$.
Let $N_{\tau_j} = 1 + \tau_j + \cdots \tau_j^{p-1}$ denote the norm of $\tau_j$.
Let $s:Q\to G$ be a set-theoretic section of \eqref{Egaloisexactintro}.  
The element $\omega \in H^2(Q, N)$ classifying \eqref{Egaloisexactintro} is determined by elements $a_j, c_{j,k} \in N$ where
$a_j = s(\tau_j)^p$ for $0 \leq j \leq r$ and where, for $0 \leq j < k \leq r$,
\[c_{j,k}=[s(\tau_k), s(\tau_j)] = s(\tau_k)s(\tau_j)s(\tau_k)^{-1}s(\tau_j)^{-1}.\] 

Here is the second main result of this paper 

\begin{theorem} \label{Tintro} (see \Cref{thm:d2Explicit})
Suppose $\phi \in H^1(N,M)^Q$ is a class represented by a homomorphism $\phi:N \to M$.
Then $\phi$ is in the kernel of $d_2$ if and only if
there exist $m_0, \ldots, m_r \in M$ such that 
\begin{enumerate}
\item $\phi(a_j)= - N_{\tau_j} m_j$ for $0 \leq i \leq r$ and
\item $\phi(c_{j,k})=(1-\tau_k)m_j - (1-\tau_j)m_k$ for $0 \leq j < k \leq r$. 
\end{enumerate}
\end{theorem}

This theorem is a consequence of the general result about $d_2$ given in \Cref{prop:kerd2v1}, combined with a direct comparison of cocycle representatives coming from two different resolutions which compute $H^*(Q,M)$. 

The last section of this paper, \Cref{finitefield}, is disjoint from the main results and is not new, but the methods use new topological tools, and are included for this reason. We recover results about the zeta function mod $p$ of the Fermat curve of exponent $p$ over a finite field of coprime characteristic.

Here is the motivation for studying the first Galois cohomology group of the relative homology 
$H_1(U,Y; \Z/p)$.
Let $X$ be a smooth, proper curve over a number field $k$ and let $\overline{b}$ be a geometric point of $X$. 
Let $\pi = \pi_1(X_{\overline{k}},\overline{b})$ 
denote the geometric \'etale fundamental group of $X$ based at $\overline{b}$, and let $$\pi = [\pi]_1 \supseteq [\pi]_2 \supseteq \ldots \supseteq  [\pi]_n \supseteq \ldots $$ denote the lower central series.
Let $G$ denote the Galois group of the maximal extension of $k$ ramified only over the primes of bad reduction for $X$, the infinite places, and a chosen prime $p$.
Using work of Schmidt and Wingberg \cite{SW92}, Ellenberg \cite{Ellenberg_2_nil_quot_pi} defines a series of obstructions to a point of the Jacobian of a curve $X$ lying in the image of the Abel-Jacobi map. Namely, $X(k)$ and $\Jac X (k)$ can be viewed as subsets of $H^1(G, \pi^{\rm ab}_p)$, where for a nilpotent profinite group, the $p$-subscript denotes the $p$-Sylow \cite[Chapter 7]{Stix}. The first of these obstructions
\[\delta_2:  H^1(G, \pi^{\rm ab}_p) \to H^2(G, ([\pi]_2/[\pi]_3)_p )\] was also studied by Zarkhin \cite{zarhin}; it  is the coboundary map associated to the $p$-part of the exact sequence 
\[ 0 \to [\pi]_2/[\pi]_3  \to \pi/[\pi]_3 \to \pi/[\pi]_2 \to 0,  \]
and has the property that $\Ker \delta_2 \supset X(k)$. Ellenberg's obstructions are related to the non-abelian Chabauty methods of \cite{Kim_Siegel} \cite{Kim_Siegel} \cite{DC-W1} \cite{B-DC-K-W}. The work of \cite{C-NGJ} gives interesting information related to the embedding $\Jac X(k) \subset H^1(G, \pi^{\rm ab}_p) $ for the Fermat curve.

To pursue this application in the case of Fermat curves, set $M=H_1(U, Y; \ZZ/p)$ and $Q=\rGal{L}{K}$.
In future work, we provide information about $N$
(the Galois group of the maximal extension of $L$ ramified only over the 
prime above $p$ and the infinite places) and the elements $a_j, c_{j,k} \in N$ which classify \eqref{Egaloisexactintro}.

In the mentioned future work, to apply \Cref{Tintro}, 
we need additional information about the elements $B_q \in \ZZ/p[\mu_p \times \mu_p]$ 
which we include in Sections \ref{Snorm} - \ref{Sinvariant}. 
Specifically, we need \Cref{Tnorm} 
which states that the norm $N_q$ of $B_q$ is zero for all $q \in Q$ and all $p \geq 5$;
\Cref{Ph1umq} which states that ${\rm codim}(H_1(U; \ZZ/p)^Q, M^Q)=2$;
and \Cref{Pkernel} which is about the kernels of $B_{\tau_j}-1$.

\section{Review and extension of previous results} \label{Ssurvey}

Throughout this paper, $p$ is an odd prime satisfying Vandiver's conjecture.

In our previous paper \cite{WINF:birs}, we extended results of Anderson \cite{Anderson} regarding the action of the absolute Galois group of a number field on the first homology of Fermat curves. 
In this section we briefly summarize and generalize these results.

The homology group associated to the Fermat curve of exponent $p$ in which one sees the Galois action most transparently is the relative homology group $H_1(U,Y; \Aa)$. 
The path $\beta: [0,1] \to U(\CC)$ given by $t \mapsto (\sqrt[p]{t},   \sqrt[p]{1-t})$, where $\sqrt[p]{-}$ denotes the real $p$th root, determines a singular $1$-simplex in $H_1(U,Y; \Aa)$ whose class we denote by the same name.
By \cite[Theorem 6]{Anderson}, $H_1(U,Y; \Aa)$ is a free rank one module with generator $\beta$ 
over the group ring 
\[\Lambda_1=\Aa[\mu_p\times \mu_p] = \Aa[\e_0,\e_1]/\langle \e_i^p-1\rangle .\] 
Note that $\Lambda_1$ itself has an action by $\Gal{\Q}$, where $g \in \Gal{\Q}$ acts on both $\e_0$ and $\e_1$ as it does on a primitive $p$-th root of unity $\zeta $ in $K=\Q(\zeta)$. 
The action of $ g\in \Gal{\Q}$ on $H_1(U,Y; \Aa)$ is twisted in the sense that
\[ g \cdot ( f(\e_0,\e_1) \beta ) =(g \cdot f(\e_0,\e_1)) (g\cdot \beta) = (g \cdot f(\e_0,\e_1)) B_g \beta. \]
In particular, if $g$ fixes $K$, it is easier to describe the action.

Further, by \cite[Section 10.5]{Anderson}, if a Galois element fixes the splitting field $L$ of $1-(1-x^p)^p$, then it acts trivially on $H_1(U,Y; \Z/p)$. Hence to determine the action of $\Gal{\Q}$, we are reduced to determining the action of the finite Galois group $\rGal{L}{\Q}$. To do this explicitly, we need to know the structure of these Galois groups; this is described in the first subsection.

The next subsection introduces the question of determining $B_q$, where $q$ is an element of the Galois group $Q := \rGal{L}{K}$.

\subsection{The Galois groups $\rGal{L}{K}$ and $\rGal{L}{\Q}$}\label{sec:Galois}

Let $r = \frac{p-1}{2}$; by \cite[Lemma 3.3]{WINF:birs}, the splitting field of $L$ of $1-(1-x^p)^p$ is
\[ L = K( \sqrt[p]{\zeta}, \sqrt[p]{1-\zeta^{-i}} | 1\leq i \leq r). \]
Let $\sigma \in \Gal{K}$; for an element $\theta$ of $K$, let $\sqrt[p]{\theta}$ be a choice of a primitive $p$-root. We define $\kappa(\theta)\sigma$ to be the element of $\Aa$ such that
\[  \sigma \cdot \sqrt[p]{\theta} = \zeta^{\kappa(\theta)\sigma}\sqrt[p]{\theta}. \]
Then $\kappa(\theta)$ defines a homomorphism $\Gal{K} \to \Aa$, which factors through $\rGal{K(\sqrt[p]{\theta})}{K}$.

From \cite[Corollary 3.7]{WINF:birs}, the map
\begin{align}\label{eq:GalLK}
C = \kappa(\zeta) \times \prod_{i=1}^r \kappa(1-\zeta^{-i}): \rGal{L}{K} \to (\Z/p)^{r+1}
\end{align}
is an isomorphism. This relationship has a geometric meaning explored further in \cite[Section 4]{WINF:birs}. We use $C$ to give a convenient basis of $Q=\rGal{L}{K}$.

\begin{definition}\label{def:Qbasis}
For $0 \leq i \leq r$, let $\tau_i$ be the inverse image under $C$ of the $i$th standard basis vector of $(\Z/p)^{r+1}$.
In other words, consider the basis for $L/K$ given by
$t_0=\sqrt[p]{\zeta}$ and $t_i=\sqrt[p]{1-\zeta^{-i}}$ for $1 \leq i \leq r$.  
Then $\tau_i$ acts by multiplication by $\zeta$ on $t_i$ and acts 
trivially on $t_j$ for $0 \leq j \leq r$, $j \neq i$.
\end{definition}

Now we turn to studying the Galois group $\rGal{L}{\Q}$; note that $L/\Q$ is itself Galois since $L$ is a splitting field. 
There is an extension
\[ 1 \to Q \to \rGal{L}{\Q} \to (\Z/p)^*  \to 1.\]
Since $ \rGal{K}{\Q} \cong (\Z/p)^*$ has order coprime to the order of $Q$, the Schur-Zassenhaus theorem implies 
that $\rGal{L}{\Q}$ splits as a semidirect product of $Q$ and $ (\Z/p)^*$. The next result determines this semidirect product.

\begin{lemma} \label{GalLQ_semidirect_product}
The extension $L/\Q$ is Galois with group $Q \rtimes_\psi (\Z/p)^*$ where 
$\psi: (\Z/p)^* \to {\rm Aut}(Q)$ is given by the conjugation action 
\[ 
\psi(a)\cdot \tau_i =
\begin{cases}
(\tau_{ia})^a, \quad \text{if } i \neq 0, \\
\tau_0,\quad \text{if } i = 0.
\end{cases}
\]
In particular, if $a$ is a generator of $(\Z/p)^*$, then $\psi(a)$ acts transitively on the set of subgroups $\langle \tau_i \rangle$ for $1 \leq i \leq r$.
\end{lemma}

\begin{proof}
We already remarked that $\rGal{L}{\Q}$ is a semi-direct product $Q \rtimes_\psi (\Z/p)^*$; we just need to determine $\psi$.
For $a \in (\Z/p)^*$, let $\alpha_a \in {\rm Aut}(K)$ be given by $\zeta \mapsto \zeta^a$.
For the case $i\neq 0$, we need to show
\[\alpha_a \tau_i \alpha_a^{-1} (z) = (\tau_{ia})^a (z), \text{\ for all }z \in L, 0 \leq i \leq r.\]
As in \cref{def:Qbasis}, let $t_j=\sqrt[p]{1-\zeta_p^{-j}}$, for $1\leq j \leq r$, and $t_0 = \sqrt[p]{\zeta} $. Since $t_j$, $0 \leq j \leq r$, generate $L$ over $K$, it suffices to check the above for $z=t_j$.

If $j=ia$, then $(\tau_{ia})^a (t_j) = \zeta^a t_j$ and 
\[\alpha_a \tau_i  \alpha_a^{-1}(t_{ia}) = \alpha_a \tau_i (t_i) = \alpha_a (\zeta t_i ) = \zeta^a t_j.\]
 If $j \neq ia$, then $t_j$ is fixed by both $\alpha_a \tau_i \alpha_a^{-1}$ and $\tau_{ia}$.

 For the case $i=0$, we need to show $\alpha_a \tau_0 \alpha_a^{-1} (t_j) = \tau_0(t_j)$, for all $0 \leq j \leq r$.
 For $j>0$, $t_j$ is fixed by both $\tau_0$ and $\alpha \tau_0 \alpha^{-1}$.
If $j=0$, then $\tau_0(t_0)= \zeta t_0$ and 
 \[\alpha_a \tau_0 \alpha_a^{-1}(t_0) = \alpha_a \tau_0 (t_0^{a^{-1}}) = \alpha_a(\zeta^{a^{-1}} t_0^{a^{-1}})= \zeta t_0.\]
\end{proof}

\subsection{Determining the action of $Q$ on $H_1(U,Y; \Aa)$}\label{sec:determineAction}

The action of $q \in Q$ on $H_1(U,Y; \Aa)$ is determined by a unit $B_q $ of $ \Lambda_1$, where $\Lambda_1 = \Aa[\mu_p \times \mu_p] \cong \Aa[\e_0, \e_1]/\langle \e_i^p-1 \rangle$. Denote by $\Lambda_0$ the group ring $\Aa[\mu_p ]  = \Aa[\e]/\langle \e^p-1\rangle$. Let $\bar\Lambda_i = \Lambda_i \otimes_{\Fp} \Fpbar$. Define a map
$ \dsh: \bar\Lambda_0^\times \to \bar\Lambda_1^\times  $
by 
\[ \dsh(u(\e)) = \frac{u(\e_0) u(\e_1)}{u(\e_0 \e_1)}.  \]
By \cite[Theorems 7 and 9]{Anderson}, $B_q$ is in the image of $\dsh$; in fact, $B_q = \dsh(\Gamma_q)$, where $\Gamma_q \in \bar\Lambda_0^\times$ is unique modulo the kernel of $\dsh$, which consists of $\e^j$, $0\leq j \leq p-1$. 
Moreover, for such a $\Gamma_q$, if we write $\Gamma_q = \sum_{i=0}^{p-1} d_i \e^i$ with $d_i \in \Fpbar$, then $\sum_{i=0}^{p-1} d_i = 1$ \cite[Lemma 5.4]{WINF:birs}.

The element $B_q$ has several nice properties; it is symmetric under the involution of $\Lambda_1$ exchanging $\e_0$ and $\e_1$. Further, by \cite[9.6, 10.5.2]{Anderson},
$B_q-1$ is in the ideal $ \langle (1-\e_0)(1-\e_1)\rangle $ of $\Lambda_1$, 
which corresponds to the homology group $H_1(U; \Aa)$ \cite[Lemma 6.1]{WINF:birs}.

As we will see shortly, the image of $\Gamma_q$ under the logarithmic derivative $\dlog: \bar\Lambda_0^\times \to \Omega(\bar\Lambda_0)$ (to the K\"ahler differentials on $\bar\Lambda_0$) gives us the information needed to determine $\Gamma_q$ and therefore $B_q$. Namely, we know from \cite[Corollary 4.2]{WINF:birs} that, modulo a term in $\Fpbar \dlog \e$,
\begin{align}\label{eq:dlogGamma}
 \dlog(\Gamma_q) = \sum_{i=1}^{p-1} c_i \e^i \dlog \e, 
 \end{align}
where $c_i = \kappa(1-\zeta^{-i})(q)$. Moreover, \eqref{eq:GalLK} along with \cite[Corollary 4.4]{WINF:birs} determines the coefficients $c_i$ from $q$. Namely, let $c_0 = \kappa(\zeta)(q)$; then $c_0, \dots c_r$ are determined by the isomorphism $C$, and for $i>r$,
\begin{align}\label{eq:cis}
 c_i  = c_{p-i} - i c_0.
\end{align}

\section{Explicit formula for the action of the Galois group} \label{Sexplicit}

In this section, we find an explicit formula for $B_q$ for each $q\in Q$, starting with the results summarized in the previous section. This is possible since 
$\Psi_q := \dlog \Gamma_q$ uniquely determines $B_q$ by \cite[10.5]{Anderson} (see also \cite[Proposition 5.1]{WINF:birs}).

\subsection{Truncated exponential maps}

Consider the group ring $\Lambda_0 \cong \Fp[\e]/(\e^p -1) $; let $y=\e-1$, so that $\Lambda_0 \cong \Fp[y]/ \langle y^p \rangle$. 
An element $f \in \Lambda_0$ (or $\bar\Lambda_0$) can be written uniquely in the form $f = \sum_{i=0}^{p-1} a_i y^i$. Let $f_y$ be the derivative of $f$ with respect to $y$.
Then $f_y (0) = a_1$. 

For $f \in y \Lambda_0$ (or $f \in y \bar\Lambda_0$), 
we define an exponential in $\Lambda_0$ by 
\begin{align} \label{E0}
E_0(f) = \sum_{i=0}^{p-1} f^i/ i!.
\end{align}
If $f,g \in y \bar\Lambda_0$, then $E_0(f)E_0(g)=E_0(f+g)$ and $E_0(f)^{-1} = E_0(-f)$.

\begin{lemma} \label{dlogEF_y} 
If $f \in y \bar\Lambda_0 $, then 
\[\dlog (E_0 (f)) = (1 + f_y (0)^{p-1} y^{p-1}) df.\]
\end{lemma}

\begin{proof}
Write $f=yg$ and note that $f_y(0)=g(0)=a_1$.  Then $f^{p-1} = y^{p-1} a_1^{p-1}$ because $y^p=0$.
So $E_0(-f) f^{p-1}  = y^{p-1} a_1^{p-1}$, again because $y^p=0$.
Hence,
\begin{align*}
\dlog(E_0(f))  &= E_0(f)^{-1} \frac{dE_0}{df} df = E_0(-f) (E_0(f) - \frac{1}{(p-1)!} f^{p-1}) df\\
& = (1 + E_0(-f) f^{p-1}) df =  (1 + f_y (0)^{p-1} y^{p-1}) df.
\end{align*}
\end{proof}

Now we move on to the group ring $\Lambda_1 = \Fp[\mu_p \times \mu_p] \cong \Fp[\e_0,\e_1]/\langle e_i^p -1 \rangle$. Let $y_i = \e_i -1$, so $\Lambda_1 = \Fp[y_0, y_1]/ \langle y_0^p, y_1^p \rangle$.

Let $\mathbb{W}$ denote the Witt vectors over $\FF_p$ (respectively $\Fpbar$). Since the characteristic of $\mathbb{W}[\frac{1}{p}]$ is zero, the usual exponential map 
\[\exp (f) = \sum_{n=0}^{\infty} \frac{f^n}{n!}\] 
is well-defined for $f \in \mathbb{W}[\frac{1}{p}][y_0,y_1]/\langle y_0^p, y_1^p \rangle$.

\begin{lemma} \label{L:nodenom}
If $f \in \langle y_0, y_1 \rangle \subset \mathbb{W}[y_0,y_1]/\langle y_0^p, y_1^p \rangle$, 
then $\exp(f) \in \mathbb{W}[y_0,y_1]/\langle y_0^p, y_1^p \rangle$.
\end{lemma}

\begin{proof}
It suffices to check that $f^n/n!$ has coefficients in $\mathbb{W}$ for each $n$.  This is clear if $n < p$.
If $n \geq p$, write $f = f_0 y_0 + f_1 y_1$ for $f_0, f_1 \in  \mathbb{W}[y_0,y_1]/\langle y_0^p, y_1^p \rangle$.
Then $f^{p}= \sum_{i=1}^{p-1} {p \choose i} f_0^i f_1^{p - i} y_0^n y_1^{p-i}$. 
Since $p \mid {p \choose i}$ for $1 \leq i \leq p-1$,
it follows that $f^p/p!$ has coefficients in $\mathbb{W}$. 
If $p < n \leq 2p-2$, then $f^n/n! = (f^p/p!) f^{n-p}/((p+1)(p+2) \cdots n)$ and so $f^n/n!$ has coefficients in $\mathbb{W}$.
If $n \geq 2p-1$, then $f^n/n! = 0$.
\end{proof}

We now define an exponential $E_1$ for $f \in \langle y_0, y_1 \rangle \subset \Lambda_1$. 
Let $\tilde{f} \in \mathbb{W}[y_0, y_1]/ \langle y_0^p, y_1^p \rangle$ be any lift of $f$; define
\begin{align} \label{E1}
E_1(f) = \overline{\exp(\tilde{f})}
\end{align}
where $\overline{\exp(\tilde{f})}$ denotes the image in $\Lambda_1$ (or $\bar \Lambda_1 $) 
of $\exp(\tilde{f})$.

\begin{lemma}\label{lem:expfun}
If $f,g \in \langle y_0, y_1 \rangle \subset \Lambda_1$ (or $\bar\Lambda_1$), then
\begin{enumerate}
\item $E_1(f)E_1(g) = E_1(f + g)$,
\item $E_1(f)^{-1} = E_1(-f)$, and
\item $E_1(f) = \sum_{i=0}^{2p-2}  f^i/ i!$.
\end{enumerate}
\end{lemma}

\begin{proof}
First, if $f,g \in \mathbb{W}[\frac{1}{p}][y_0,y_1]/\langle y_0^p, y_1^p \rangle$, 
then $\exp(f + g) = \exp(f) \exp(g)$.
By \cref{L:nodenom}, if $f \in \langle y_0, y_1 \rangle$, then 
$\exp (f) \in \mathbb{W}[y_0,y_1]/\langle y_0^p, y_1^p \rangle$. 
Thus $\exp (f+g)$, $\exp(f)$, and $\exp(g)$ are in $\mathbb{W}[y_0,y_1]/\langle y_0^p, y_1^p \rangle$. Reducing mod $p$ shows that $E_1(f) E_1(g) = E_1(f+g)$.

Next, $E_1(f)$ is invertible because $E_1(f)=1 + N$ for some element $N$ of the nilradical. 
Then $E_1(f)^{-1} = E_1(-f)$ because $E_1(f) E_1(-f) = E_1(0) = 1$. 

The last statement follows from the fact that $f^{2p-1}=0$.
\end{proof}

\subsection{$\Gamma_q$ from $\Psi_q$}

In this subsection, we determine a formula for $\Gamma_q$ in terms of $\Psi_q = \dlog \Gamma_q$. For convenience, we drop the subscript $q$, but 
everything depends on this chosen element of $Q$.

\begin{proposition}\label{Gamma_from_c}
Write \[\Psi = \sum_{i=1}^{p-1} c_i \e^i \dlog \e,\] 
and let $c = \sum_{i=1}^{p-1} c_i$ be its coefficient sum. Let $F \in \Fpbar$ be a solution to the equation $F^p - F + c = 0$, and define
\begin{equation}\label{Egamma}
\gamma(\e) = \sum_{i=1}^{p-1} \left(\frac{c_i + c -F}{i}\right)\e^i - \sum_{i=1}^{p-1} \frac{c_i}{i}.
\end{equation}
Then \[\Gamma =E_0 (\gamma(\e)).\]
\end{proposition}

\begin{proof}
By \eqref{eq:dlogGamma} (and \cite[Corollary 4.2]{WINF:birs}), 
$\dlog \Gamma = \Psi $ modulo $\Fpbar \dlog \e$.
We rewrite $\Psi$ in the nilpotent basis, i.e.,
\[\Psi = \sum_{i=1}^{p-1} c_i \e^i \dlog \e = \sum_{i=1}^{p-1} c_i (y+1)^{i-1} dy.\]
To find a solution to $\Psi = \dlog (\Gamma)$, we find $f \in y \bar\Lambda_0 $ such that 
$\Gamma = E_0(f)$; any unit in $\Lambda_0$ is of this form up to scaling.

From the congruence $\binom{p-1}{i}  \equiv (-1)^i \bmod p$, it follows that
\begin{align}
y^{p-1} = ((y+1) - 1)^{p-1} = \sum_{i=0}^{p-1} \binom{p-1}{i} (y+1)^i (-1)^{p-1-i} = \sum_{i=0}^{p-1} (y+1)^i.
\end{align}
By \cref{dlogEF_y}, 
\[\dlog (E_0(f)) =  (1 + f_y (0)^{p-1} y^{p-1}) df= df + f_y (0)^{p} \left(\sum_{i=0}^{p-1} (y+1)^{i}\right) dy.\]
Define $f_i\in \Fpbar$ by $f = \sum_{i=0}^{p-1} f_i (y+1)^i$, and note that $f_y (0) = \sum_{i=0}^{p-1} i f_i$. 
For $1 \leq i \leq p-1$, we need to solve the equation 
\[i f_i + \left(\sum_{i=0}^{p-1} i f_i\right)^p = c_i\] 
in such a way that $\sum_{i=0}^{p-1} f_i = 0$. This last condition comes from the fact that $\sum_{i=0}^{p-1} d_i = 1$ if $\Gamma = \sum_{i=0}^{p-1} d_i \e^i$, 
\cref{sec:determineAction} (or \cite[Lemma 5.4]{WINF:birs}).

Adding the first set of equations gives 
\[c:=\sum_{i=1}^{p-1} c_i = (p-1)\left(\sum_{i=0}^{p-1} i f_i\right)^p + \sum_{i=0}^{p-1} i f_i.\]

Let $F = \sum_{i=0}^{p-1} i f_i$; then $F$ is a solution of $F^p - F + c = 0$. 
Choose any of the $p$ solutions $F, F+1, \dots , F+ (p-1)$ in $\Fpbar$.
Then $f_i = (c_i + c - F)/i$ for $i> 0$ and $f_0 = - \sum_{i>0} f_i = - \sum c_i /i$.
\end{proof}

\subsection{$B_q$ from $\Psi_q$}

In this section, we determine a formula for $B$ in terms of $\Psi$.
Let $\gamma_i = \gamma(\e_i)$ for $i=0,1$ and let $\gamma_{01}=\gamma(\e_0\e_1)$, 
where 
\begin{equation}\label{Egamma}
\gamma(\e) = \sum_{i=1}^{p-1} (\frac{c_i + c -F}{i})\e^i - \sum_{i=1}^{p-1} \frac{c_i}{i}.
\end{equation}

\begin{theorem}\label{cor:Beta}
Suppose $p$ is an odd prime satisfying Vandiver's conjecture. 
The action of $q \in Q={\rm Gal}(L/K)$ 
on the relative homology $H_1(U, Y; \Aa)$ of the Fermat curve is determined by
the element $B_q \in \Lambda_1$ with the explicit formula
\[B_q =\frac{E_0(\gamma_0) E_0(\gamma_1)}{E_0(\gamma_{01})} = \frac{E_1( \gamma_0 + \gamma_1)}{ E_1(\gamma_{01} )  - T},\] where $T$ is the ``error term"
\[ T = E_{1}(\gamma_{01}) - E_{0}(\gamma_{01})=\sum_{i=p}^{2p-2} \frac{\gamma_{01}^i}{i!}. \]
\end{theorem}

\begin{proof}
By \cite[Section 8.4]{Anderson}, $B = \Gamma(\e_0)\Gamma (\e_1)/ \Gamma(\e_0 \e_1)$ in $\Lambda_1$.
By \Cref{Gamma_from_c}, $\Gamma(\e) = E_{0} (\gamma(\e))$. 
If $i=0,1$, then $E_{0}(\gamma_i) = E_{1}(\gamma_i)$ since $\gamma_i^p = 0$. 
By \Cref{lem:expfun}, $\Gamma (\e_0)\Gamma (\e_1) = E_1(\gamma_0+\gamma_1)$.
Since $\gamma_{01}^p$ is not necessarily zero, the error term $T$ appears in the denominator.
\end{proof}

\begin{remark}\label{rem:Tideal}
The error term $T$ is in the ideal $\langle y_0,y_1\rangle^p$ since $\gamma_{01} \in \langle y_0,y_1\rangle$. 

In the atypical situation that $\gamma_{01}^p=0$, then $T=0$ and $B_q = E_1(\gamma_0 + \gamma_1 - \gamma_{01})$.\end{remark}

The next formula follows immediately from \cref{cor:Beta}.
\begin{equation} \label{Bqinverse}
B_{q^{-1}} = E_1(\gamma_{01} -\gamma_0 -\gamma_1) - E_1(-\gamma_0 -\gamma_1)T. 
\end{equation}

For better display in the next examples, let $x = \e_0 -1$ and $y=\e_1-1$. We arrived at the formulas using Magma; 
it is difficult to do these calculations by hand.

\begin{example}\label{ex:p3Bq}
Let $p=3$. Then $Q = \langle \tau_0, \tau_1 \rangle = (\Z/3)^2$.

If $q=\tau_0$, then $c_0=1$, $c_1=0$, and $c_2=1$; hence $c=1$. 
Let $F$ be a solution of $F^3-F+1=0$, so $f_0 = 1$, $f_1=1-F$, and $f_2 = 1+F$.  Then
\[\gamma_{\tau_0} = 1+(1-F)\e + (1+F)\e^2 = Fy + (1+F)y^2.\]

If $q=\tau_1$, then $c_0=0$ and $c_1 = c_2 =1$; hence $c=-1$. 
Let $F$ be a solution to $F^3-F-1=0$, so that $f_0=0$, $f_1 = -F$, and $f_2 = F$.  Then
\[\gamma_{\tau_1} = F(\e^2-\e) = F(y+y^2).\]
After a calculation, one obtains that
\[ B_{\tau_0} = 1+ xy + 2xy(x+y) \ {\rm and} \ B_{\tau_1} = 1+ 2xy(x+y) + x^2y^2. \]
\end{example}

\begin{example}
Let $p=5$; then $Q=\langle \tau_0, \tau_1, \tau_2\rangle \simeq (\ZZ/5)^3$, and we have:
\begin{eqnarray*}
B_{\tau_0} - 1 & = & 4x^4y^4 + x^4y^3 + 3x^4y^2 + 4x^4y + x^3y^4 + x^3y^3 + 2x^3y^2 + 4x^3y\\
& + & 3x^2y^4 +  2x^2y^3 + 3x^2y + 4xy^4 + 4xy^3 + 3xy^2;\\
B_{\tau_1} -1 & = & 2x^4y^4 + 2x^4y^3 + 4x^4y^2 + 4x^4y + 2x^3y^4 + 2x^3y^3 + 4x^3y^2 + x^3y\\
& + & 4x^2y^4 + 4x^2y^3 + x^2y^2 + 4x^2y + 4xy^4 + xy^3 + 4xy^2;\\
B_{\tau_2}-1 & = & 2x^4y^4 + 3x^4y^3 + 3x^4y^2 + 3x^3y^4 + 4x^3y^3 + 4x^3y^2 + 4x^3y\\
& + & 3x^2y^4 + 4x^2y^3 + 4x^2y^2 + x^2y + 4xy^3 + xy^2.
\end{eqnarray*}
\end{example}

\begin{example}
Let $p=7$; then $Q=\langle \tau_0, \tau_1, \tau_2, \tau_3 \rangle \simeq (\ZZ/7)^4$, and we have:
\begin{eqnarray*}
B_{\tau_0}-1 & = &  x^6y^5 + 3x^6y^4 + 2x^6y^3 + 2x^6y^2 + 6x^6y \\
&+& x^5y^6 + 2x^5y^5 +  x^5y^4 + 4x^5y^3 + 6x^5y \\
&+& 3x^4y^6 +  x^4y^5 + 5x^4y^4 + 2x^4y^2\\
& +& 2x^3y^6 + 4x^3y^5 + 4x^3y^2 + 4x^3y \\
&+& 2x^2y^6 + 2x^2y^4 + 4x^2y^3 + 4x^2y^2 + 3x^2y\\
& +& 6xy^6+ 6xy^5 + 4xy^3 + 3xy^2;\\
B_{\tau_1}-1 & = & 5x^6y^6 + 3x^6y^5 + 2x^6y^4 + 3x^6y^3 + 6x^6y^2 + 6x^6y\\
& + & 3x^5y^6 + 3x^5y^5 + 4x^5y^4 + 4x^5y^3 + 5x^5y^2 + x^5y\\
& + & 2x^4y^6 + 4x^4y^5 + x^4y^4 + 4x^4y^3 + 5x^4y^2 + 6x^4y\\
& + & 3x^3y^6 + 4x^3y^5 + 4x^3y^4 + 2x^3y^3 + 6x^3y^2 + x^3y\\
& + & 6x^2y^6 + 5x^2y^5 + 5x^2y^4 + 6x^2y^3 + x^2y^2 + 6x^2y\\
& + & 6xy^6 + xy^5 + 6xy^4 + xy^3 + 6xy^2;\\
B_{\tau_2}-1 & = & 2x^6y^6 + 6x^6y^5 + 5x^6y^4 + x^6y^3\\
&  + & 6x^5y^6 + x^5y^5 + 5x^5y^4 + 2x^5y^3 + 3x^5y^2 + 6x^5y \\
& + & 5x^4y^6 + 5x^4y^5 +  4x^4y^4 + 5x^4y^2 + 2x^4y\\ 
&+ & x^3y^6 + 2x^3y^5 + 3x^3y^3 + x^3y^2 + 4x^3y \\
& + & 3x^2y^5 + 5x^2y^4 + x^2y^3 + 4x^2y^2 + 3x^2y\\
& + & 6xy^5 + 2xy^4 + 4xy^3 + 3xy^2;\\
B_{\tau_3}-1&=&4x^6y^5 + 2x^6y^3 + 4x^6y^2 \\
& + & 4x^5y^6 + 4x^5y^5 + x^5y^4 + 6x^5y^3 + 3x^5y^2\\
& + & x^4y^5 + 4x^4y^4 + 5x^4y^3 + 4x^4y^2 + 6x^4y\\ 
& + & 2x^3y^6 + 6x^3y^5 + 5x^3y^4 + 2x^3y^3 + 2x^3y\\
& + & 4x^2y^6 + 3x^2y^5 + 4x^2y^4 + 2x^2y^2 + 5x^2y\\
& + & 6xy^4 + 2xy^3 + 5xy^2.
\end{eqnarray*}
\end{example}

\section{Norm equalities for general primes} \label{Snorm}

For $q \in Q$, consider the unit $B_q$ in $\Lambda_1 = \Z/p[\e_0,\e_1]/\langle \e_i^p-1\rangle$. 
Note that $B_q^p=1$ since $q$ has order $p$. In \Cref{SSnorm}, we strengthen this by proving that the norm 
\[ N_q:= 1+B_q+\dots +B_q^{p-1}\]
is zero, except in the special case that $p=3$ and $q$ does not fix $\zeta_9 \in L$.
In \Cref{application2}, we study the power of $B_q-1$ which trivializes $H(U; \ZZ/p)$.

Throughout this section, it is again more convenient to work with the nilpotent basis of $\Lambda_1$ given by $y_i=\e_i-1$, 
so that $\Lambda_1=\Aa[y_0,y_1]/\langle y_0^p,y_1^p\rangle$. 

\subsection{Vanishing norms} \label{SSnorm}

Before studying the norm of $B=B_q$, we need an auxiliary result.
Write \[\tilde \gamma = \gamma_0+\gamma_1 - \gamma_{01},\] 
where $\gamma$ is as defined in \eqref{Egamma}, 
$\gamma_i = \gamma(\e_i)$ for $i=0,1$, and $\gamma_{01}=\gamma(\e_0\e_1)$.  Note that $\tilde \gamma \in \langle y_0, y_1 \rangle$, since $\gamma \in \langle y \rangle \subset \bar\Lambda_0$.

\begin{proposition}\label{prop:tildeideal}
If $q \in Q$, then $\tilde \gamma$ is in the ideal $\langle y_0,y_1\rangle^2$. If $p \geq 5$, or if $p=3$ and $q$ fixes $\zeta_9 \in L$, then $\tilde \gamma$ is in $\langle y_0,y_1\rangle^3$. More precisely,
\begin{enumerate}
\item $\tilde \gamma = y_0 y_1 \eta$ for some $\eta \in \bar\Lambda_1$;
\item and $\tilde \gamma \equiv \alpha y_0 y_1 (y_0 + y_1)$ modulo $\langle y_0,y_1\rangle^4$, 
for some constant $\alpha \in \FF_p$, unless $p = 3$ and $q \not \in \langle \tau_1 \rangle$.
\end{enumerate}
\end{proposition}

\begin{proof}
For part (1), 
suppose $\gamma=\sum_{i=0}^{p-1} a_i y^i$. Then 
\[\tilde \gamma = \gamma(\e_0) + \gamma(\e_1) - \gamma(\e_0 \e_1) = \sum_{i=0}^{p-1} a_i(y_0^i+y_1^i) - \sum_{i=0}^{p-1} a_i(y_0+y_1+y_0y_1 )^i. \]
Consider the coefficient of $y_0^k$ (equivalently, $y_1^k$) in 
\[\gamma(\e_0 \e_1) =\sum_{i=0}^{p-1} a_i(y_0+y_1(1+y_0))^i = \sum_{i=0}^{p-1} \sum_{j=0}^i a_i \binom{i}{j}y_0^j y_1^{i-j}(1+y_0)^{i-j}  . \]
The monomial $y_0^k$ appears in this sum only when $i=j$, hence also $j=k$, and the coefficient is thus $a_j$. 
It follows that the coefficients of $y_0^k$ and $y_1^k$ in $\tilde \gamma$ are zero, so $\tilde \gamma$ is divisible by $y_0y_1$.

For part (2), note that 
$\tilde \gamma = y_0y_1 \eta$, for some $\eta \in \bar\Lambda_1$, by part (1). The constant coefficient $w$ of $\eta$ equals the coefficient of $y_0y_1$ in $-\gamma_{01}$.
Write 
\[\gamma = \sum_{i=0}^{p-1} f_i \epsilon^i = \sum_{i=0}^{p-1} f_i (y+1)^i;\]
then
\begin{equation}\label{Egamma01}
- \gamma_{01} = - \sum_{i=0}^{p-1} f_i (y_0+1)^i(y_1+1)^i,
\end{equation}
so it follows that
\[w= - \sum_{i=1}^{p-1} f_i i^2.\]
Since $f_i=\frac{c_i + c - F}{i}$, this simplifies to
\[w= - (c-F)  \sum_{i=1}^{p-1} i - \sum_{i=1}^{p-1} i c_i= - \sum_{i=1}^{p-1} i c_i .\]
In particular, this proves that the assignment $\vec{c} = (c_0,c_1,\dots, c_{p-1}) \to w$ is linear.  

Case 1: If $c_0=0$ (equivalently, if $q$ fixes $\zeta_{p^2}$), then $c_{p-i}=c_i$. In this case, $w= - \sum_{i=1}^{(p-1)/2} c_i (i + (p-i))=0$.

Case 2: Suppose $c_0=1$ and $c_i = 0$ for $1 \leq i \leq r=\frac{p-1}{2}$.  Then $c_{p-j}=j$ for $1 \leq j \leq r$.
So \[w= - \sum_{i=r+1}^{p-1} i (p-i) = - \sum_{j=1}^{r} (p-j)j =\sum_{j=1}^{r} j^2,\] and $w = r(r+1)(2r+1)/6$.
If $p \geq 5$, then this gives $w=0$.  

General case:   
Since $\vec{c} \to w$ is linear, the above two cases prove that $w=0$ for all $q$ when $p \geq 5$.
Finally, $\eta \equiv \alpha (y_0 + y_1)$ modulo $\langle y_0, y_1 \rangle^2$ since it is symmetric with respect to the involution switching $y_0$ and $y_1$.
\end{proof}

The following consequence of \Cref{prop:tildeideal} will be used in \cref{Sinvariant}.

\begin{corollary}\label{cor:BqinI3}
Suppose $p\geq 5$. 
\begin{enumerate}
\item Then $B_q -1$ is in the ideal $\langle y_0,y_1 \rangle^3$ for all $q \in Q$. 
In fact, for some constant $\alpha \in \FF_p$, there is a congruence 
$B_q-1 \equiv \alpha y_0y_1(y_0 + y_1)$ modulo $\langle y_0,y_1 \rangle^4$.

\item The coefficient $\alpha$ of $y_0^2 y_1^1$ in $B_q -1$ is non-zero for all $q \in Q$ not in a linear subspace of 
codimension $1$.
\end{enumerate}

\end{corollary}

\begin{proof}
It suffices to show the conclusions for $B_q^{-1} -1$.
By \eqref{Bqinverse},
\[B_q^{-1} = E_1(-\tilde\gamma) - E_1(-\gamma_0-\gamma_1)T.
\]
Now $T \in \langle y_0,y_1 \rangle^p$ by \cref{rem:Tideal} so $B_q^{-1} -1 \equiv E_1(-\tilde\gamma) -1$ modulo $\langle y_0,y_1 \rangle^p$.
Furthermore, $- \tilde{\gamma} \equiv \alpha y_0 y_1 (y_0 + y_1)$ modulo $\langle y_0,y_1 \rangle^4$ by \cref{prop:tildeideal}.
By definition, $E_1(f) = \sum_{i=0}^{2p-2} f^i/i!$.
Thus 
\[E_1(-\tilde\gamma)-1 = - \tilde{\gamma} + \tilde{\gamma}^2/2 + \cdots \equiv - \tilde \gamma \bmod \langle y_0,y_1 \rangle^8.\]
Thus $B_q^{-1} -1 \equiv \alpha y_0 y_1 (y_0 + y_1)$ modulo $\langle y_0,y_1 \rangle^4$, finishing item (1).

For item (2), recall that $-\tilde{\gamma} = \gamma_{01} - \gamma_0 - \gamma_1$.
Thus $\alpha$ is the coefficient of $y_0^2 y_1^1$ in $\gamma_{01}$, 
because $\gamma_0$ and $\gamma_1$ have no terms divisible by $y_0y_1$.
As in \eqref{Egamma}, $\gamma(\epsilon) = \sum_{i=1}^{p-1} f_i \epsilon^i$ where $f_i = (c_i + c - F)/i$.
By \eqref{Egamma01},
\[\gamma_{01} = \sum_{i=1}^{p-1} f_i (y_0+1)^i (y_1+1)^i =  
\sum_{i=1}^{p-1} f_i(1+i y_0 + \binom{i}{2} y_0^2 + \cdots)(1+i y_1 + \binom{i}{2} y_1^2 + \cdots).\]
So the coefficient $\alpha$ of $y_0^2y_1$ in $\gamma_{01}$ is 
\[\alpha = \sum_{i=2}^{p-1} f_i \binom{i}{2} i=  \sum_{i=2}^{p-1} (c_i + c - F)\binom{i}{2}.\]

The centered octagonal pyramid number formula is 
$\sum_{i=2}^{p-1} \binom{i}{2} = n(4n^2-1)/3$ where $n=(p-1)/2$.  Then $4n^2 -1 \equiv 0 \bmod p$, 
so $(c-F) \sum_{i=2}^{p-1} \binom{i}{2} \equiv 0 \bmod p$.
Thus \[\alpha =  \sum_{i=2}^{p-1} c_i \binom{i}{2}.\]
Item (2) follows since the coefficient $\alpha$ is linear in $\vec{c}$ and 
does not vanish when $c_2 = c_{p-2}=1$ and all other $c_i=0$.
\end{proof}

\begin{proposition} \label{Nsimplify}
Let $N_{q^{-1}}$ be the norm of $B_{q^{-1}}$ and $\tilde \gamma = \gamma_0+\gamma_1 - \gamma_{01}$.
Then 
\[N_{q^{-1}}= N_{E_1(-\tilde{\gamma})}:= \sum_{i=0}^{p-1} E_1(-\tilde\gamma)^i.\] 
\end{proposition}

\begin{proof}
By \eqref{Bqinverse}, $B_{q^{-1}} = E_1(\gamma_{01} -\gamma_0 -\gamma_1) - E_1(-\gamma_0 -\gamma_1)T$. 
By \Cref{rem:Tideal}, $T^2=0$. Therefore, using \cref{lem:expfun} repeatedly, we have
\begin{align*}
 N_{q^{-1}} & = \sum_{m=0}^{p-1} \left(E_1(-\tilde \gamma ) - E_1(-\gamma_0-\gamma_1)T\right)^m  \\
 &=  \sum_{m=0}^{p-1} \sum_{k=0}^{m} (-1)^k\binom{m}{k} E_1(-(m-k)\tilde\gamma) E_1(-k(\gamma_0+\gamma_1)) T^k \\
 &= \sum_{m=0}^{p-1}E_1(-m\tilde \gamma ) - \sum_{m=1}^{p-1} m E_1( (1-m)\tilde\gamma - \gamma_0 - \gamma_1 ) T\\
 &=N_{E_1(-\tilde\gamma)} - \frac{T}{E_1(\gamma_{01})} \sum_{m=1}^{p-1} m E_1(-m\tilde\gamma). 
\end{align*}

To finish the proof, it suffices to show that the second term in the sum is $0$ in $\Lambda_1$.
By \Cref{prop:tildeideal}, $\tilde \gamma \in \langle y_0,y_1\rangle^2$.
Since $T \in \langle y_0,y_1\rangle^p$, it suffices to show that
\[ S = S(\tilde{\gamma}) = \sum_{m=1}^{p-1} m E_1(-m \tilde \gamma)\]
is in the ideal $I=\langle y_0,y_1\rangle^{p-1}$.
By \cref{lem:expfun}(3),
\[ S = \sum_{m=1}^{p-1} \sum_{t=0}^{2p-2} (-1)^t \frac{m^{t+1}\tilde{\gamma}^t}{t!}. \]
If $t \geq \frac{p-1}{2}$, then $\tilde{\gamma}^t \in I$.
Thus, modulo $I$, 
\[S \equiv  \sum_{t=0}^{(p-3)/2}  (-1)^t \frac{\tilde{\gamma}^t}{t!} (\sum_{m=1}^{p-1} m^{t+1}).\]
However, $\sum_{m=1}^{p-1} m^{t+1}=0$ when $0 \leq t \leq (p-3)/2$.
\end{proof}

 \begin{lemma}\label{lem:Enorm}
Suppose $f \in \Lambda_1$ is in the ideal $\langle y_0,y_1\rangle$. Then
\[N_{E_1(f)}:=\sum_{i=0}^{p-1} E_1(f)^i = f^{p-1} - \frac{f^{2p-2}}{(2p-2)!}.\]
\end{lemma}
\begin{remark}
Even though it is not possible to divide by $p$, the expression $\frac{f^{2p-2}}{(2p-2)!}$ is well-defined for $f\in \langle y_0,y_1\rangle$. 
\end{remark}

\begin{proof}
By \Cref{lem:expfun},
\[
N_{E_1(f)} = \sum_{i=0}^{p-1} E_1(f)^i = \sum_{i=0}^{p-1} E_1(if) = 1+\sum_{i=1}^{p-1} \sum_{m=0}^{2p-2} \frac{i^mf^m}{m!}.\]
Thus 
\[N_f =1 + \sum_{m=0}^{2p-2} \frac{f^m}{m!}\Big( \sum_{i=1}^{p-1} i^m \Big).\]

Recall that, modulo $p$, $\sum_{i=1}^{p-1} i^m = 0$ unless 
$m \equiv 0 \bmod p-1$ in which case $\sum_{i=1}^{p-1} i^m = -1$.
Also $(p-1)! = -1 $.
Thus
\[N_{E_1(f)} = 1 -\Big(1 + \frac{f^{p-1}}{(p-1)!} +\frac{f^{2p-2}}{(2p-2)!}\Big) = f^{p-1} - \frac{f^{2p-2}}{(2p-2)!}.\]
\end{proof}

\begin{theorem} \label{Tnorm}
For any $q\in Q$, the norm $N_q$ of $B_q$ equals $\tilde\gamma^{p-1}$. 
In particular, $N_q=0$ for all $q \in Q$ if $p\geq 5$; when $p=3$, then $N_q=0$ if $q$ fixes $\zeta_9$.
\end{theorem}

\begin{proof}
The norm of $B_q$ equals the norm of $B_q^{-1}=B_{q^{-1}}$, which is $N_{q^{-1}}$. 
By \cref{Nsimplify}, $N_{q^{-1}}= N_{E_1(-\tilde{\gamma})}$, and by \Cref{lem:Enorm},
\[ N_{E_1(-\tilde{\gamma})}  = (-\tilde\gamma)^{p-1} - \frac{(-\tilde\gamma)^{2p-2}}{(2p-2)!}.\]
From \Cref{prop:tildeideal}, $\tilde\gamma^{2p-2}$ is in the ideal $\langle y_0,y_1\rangle^{2(2p-2)}$, hence zero. Moreover, by \Cref{prop:tildeideal}(2) if $p\geq 5$, or if $q$ fixes $\zeta_{p^2}$, then $\tilde\gamma^{p-1}=0$.
\end{proof}

\begin{example}
Let $p=3$, and $q=\tau_1$;
as seen in \cref{ex:p3Bq}, $\gamma_{\tau_1} = F(\e^2-\e)$, so
\[\tilde \gamma_{\tau_1} = F(\e_0^2 -\e_0 +\e_1^2 - \e_1 - \e_0^2\e_1^2 +\e_0\e_1 )= - y_0^2 y_1^2 + y_0y_1(y_0 + y_1).  \]
Thus $\tilde\gamma_{\tau_1} \in \langle y_0,y_1\rangle^3$
and $N_{\tau_1}=\tilde\gamma_{\tau_1}^2=0$.
\end{example}

\begin{example}
Let $p=3$, and $q=\tau_0$;
as seen in \cref{ex:p3Bq},
\[ \gamma_{\tau_0} = 1+(1-F)\e + (1+F)\e^2= Fy + (1+F)y^2.\]
This implies that
\[ \tilde\gamma_{\tau_0} = y_0y_1 + (1+F) y_0y_1 (y_0 + y_1 - y_0y_1),\]
showing that $N_{\tau_0} = \tilde\gamma_{\tau_0}^2 = y_0^2 y_1^2$, which is not zero.
\end{example}

\begin{example}
Let $p=5$. Then modulo $\langle y_0, y_1 \rangle^4$:
\[\tilde\gamma_{\tau_0 } \equiv 3y_0y_1(y_0 + y_1), \ \tilde\gamma_{\tau_1 } \equiv 4y_0y_1(y_0 + y_1), \ \tilde\gamma_{\tau_2 } \equiv y_0y_1(y_0 + y_1).\]
\end{example}

\subsection{A second application}

Let $(y_0y_1) \Lambda_1 = (\epsilon_0-1)(\epsilon_1-1) \Lambda_1$ denote the augmentation ideal.
By \cite[Proposition 6.2]{WINF:birs}, the homology $H_1(U; \ZZ/p)$
can be identified with $(y_0y_1) \Lambda_1 \beta$.
In \cite[9.6 and 10.5.2]{Anderson}, for each $q \in Q$, Anderson proves that
$B_q - 1 \in (y_0y_1) \Lambda_1$;
this implies that $H_1(U; \Aa)$ is trivialized by 
the product $\prod_{i=1}^{p-1} (B_{q_i} - 1)$ for any $q_1, \ldots, q_{p-1} \in Q$.
The improvement in \Cref{cor:BqinI3} allows us to show that in fact $H_1(U; \Aa)$ is trivialized 
by the product of only $s = \lfloor 2p/3 \rfloor$ such terms when $p \geq 5$.

\begin{corollary} \label{application2}
Let $p \geq 5$ and $s = \lfloor 2p/3 \rfloor$ and $s'= \lfloor (2p+1)/3 \rfloor$.
If $T \geq s$ (resp.\ $T \geq s'$) and $q_1, \ldots, q_T \in Q$, 
then $\prod_{i=1}^{T} (B_{q_i} - 1)$ trivializes $H_1(U; \ZZ/p)$ (resp.\ $H_1(U, Y; \ZZ/p)$).
\end{corollary}

\begin{proof}
If $p \geq 5$, then  \Cref{cor:BqinI3} shows that each monomial in $B_q-1$ is a multiple of either
$y_0^2y_1$ or $y_0y_1^2$ or both.  After taking the product of $T$ such terms, 
each monomial is of the form 
$y_0^{2a+b}y_1^{a+2b} = y_0^Ty_1^Ty_0^ay_1^b$ for some $a,b \geq 0$ such that $a+b=T$.
The monomial which is least likely to be zero in $\Lambda_1$ is: 
$(y_0y_1)^{3T/2}$ when $T$ is even and $a=b=T/2$;
or $y_0^{(3T-1)/2} y_1^{(3T+1)/2}$ when $T$ is odd and $a=(T-1)/2$ and $b=(T+1)/2$ 
(or its permutation under the transposition of $y_0$ and $y_1$).
To trivialize $H_1(U; \ZZ/p)$, it suffices to trivialize $y_0y_1 \cdot \beta$, 
which is guaranteed when $3T/2 \geq p-1$ for $T$ even and when $(3T+1)/2 \geq p-1$ when $T$ is odd. 
The smallest such value is $s$.
To trivialize $H_1(U,Y; \ZZ/p)$ it suffices to trivialize 
$1 \cdot \beta$, which is guaranteed when $3T/2 \geq p$ for $T$ even 
and when $(3T+1)/2 \geq p$ when $T$ is odd. 
The smallest such value is $s'$.
\end{proof}

\section{The $Q$-invariants} \label{Sinvariant}

Let $M$ denote the homology group $H_1(U,Y; \Aa)$, which can be identified with $\Lambda_1$.  Under this identification, the homology group $H_1(U; \Aa)$ corresponds to the ideal $\langle (1-\e_0)(1-\e_1)\rangle$ \cite[Lemma 6.1]{WINF:birs}. 
Recall that $y_i=\e_i-1$.

The $Q$-invariants of $M$ are 
\[M^Q = \{m \in M \mid B_q m = m \ {\rm for \ all \ }q \in Q\}.\]
In \Cref{SQinvariant1}, we prove that ${\rm codim}(H_1(U)^Q, M^Q)=2$ for all odd $p$ and 
construct a subspace of $M^Q$ of dimension $2p+1$ for $p \geq 5$.
In \Cref{SQinvariant2}, we compare the $B_q$-invariant subspaces of $M$ for various $q \in Q$.

\subsection{A subspace of $M^Q$} \label{SQinvariant1}

For $0 \leq k \leq p-1$, define $\eta_k = \epsilon_1^{k} \sum_{i=0}^{p-1} \epsilon_0^i$
and $\gamma_k = \epsilon_0^k \sum_{i=0}^{p-1} \epsilon_1^i$.
Note that $(1-\epsilon_0) \eta_k = (1-\epsilon_1) \gamma_k=0$.

\begin{lemma}\label{lem:InvSubspace}
Let $L = \langle \eta_k, \gamma_k \rangle_{k=0}^{p-1}$, viewed as a  $\Aa$-subspace of $M$.  Then:
\begin{enumerate}
\item ${\dim}(L) = 2p-1$;
\item  ${\rm codim}(L \cap H_1(U), L) = 2$;
\item a basis for $L$ is $\{y_0^{i_0} y_1^{i_1} \mid \ {\rm at \ least \ one \ of} \ i_0, i_1 \ {\rm equals} \ p-1 \}$;
\item and $L \subset M^Q$.
\end{enumerate}
\end{lemma}

\begin{proof}
\begin{enumerate}
\item The elements $\eta_k$ for $0 \leq k \leq p-1$ generate a $\ZZ/p$-vector space of dimension $p$.
Similarly, $\gamma_k$ for $0 \leq k \leq p-1$ generate a $\ZZ/p$-vector space of dimension $p$.
The intersection $\langle \eta_k \rangle \cap \langle \gamma_k \rangle$ has dimension $1$ with basis 
$\sum_{k=0}^{p-1} \gamma_k = \sum_{k=0}^{p-1} \eta_k$.  Thus 
${\rm dim}(L) =2p-1$.
 
\item A basis for $L$ is given by $\eta_k$ for $0 \leq k \leq p-1$ and $\gamma_k$ for $0 \leq k \leq p-2$.
Write an element $\xi \in L$ in the form $\xi=A+B$ where $A = \sum_{k=0}^{p-1} a_k \eta_k$ and 
$B= \sum_{k=0}^{p-2} b_k \gamma_k$.

Since $A \in \langle 1-\epsilon_0 \rangle$, then $\xi \in \langle 1-\epsilon_0 \rangle$ if and only if 
$B \in \langle 1-\epsilon_0 \rangle$. 
Since $B=(\sum_{i=0}^{p-1} \epsilon_1^i) \sum_{k=0}^{p-2} b_k \epsilon_0^k$, this condition is satisfied if and only if (i) $\sum_{k=0}^{p-2} b_k = 0$.
Similarly, $B \in \langle 1-\epsilon_1 \rangle$, so $\xi \in \langle 1-\epsilon_1 \rangle$ if and only if 
$A \in \langle 1-\epsilon_1 \rangle$. This condition is satisfied if and only if (ii) $\sum_{k=0}^{p-1} a_k = 0$.
Since conditions (i) and (ii) are linearly independent, ${\rm codim}(L \cap H_1(U), L) = 2$.

\item This follows from the fact that $\eta_k = \epsilon_1^{k} \sum_{i=0}^{p-1} \epsilon_0^i = (y_1+1)^k y_0^{p-1}$
and $\gamma_k = \epsilon_0^k \sum_{i=0}^{p-1} \epsilon_1^i = (y_0+1)^k y_1^{p-1}$.

\item To show $L \subset M^Q$, it suffices to show that $(B_q-1) m = 0$ for each $m \in L$.
By part (3) and symmetry, it suffices to show that $(B_q-1) y_0^{i_0} y_1^{p-1} =0$.
This is true since $B_q-1 \in H_1(U) = \langle y_0 y_1 \rangle$ for all $q \in Q$, by \cref{cor:BqinI3}.
\end{enumerate}
\end{proof}

\begin{proposition} \label{Ph1umq}
If $p$ is odd, then $H_1(U)^Q$ has codimension $2$ in $M^Q$.
\end{proposition}

\begin{proof}
The result is true for $p=3$ by explicit computation.
If $p \geq 5$, then ${\rm codim}(H^1(U)^Q, M^Q) \geq 2$, since neither $y_0^{p-1}$ and $y_1^{p-1}$ are in $H_1(U)$,
but they are linearly independent in $M^Q$.
It suffices to show that the image of the map $\psi: M^Q \to \left( M/H_1(U) \right)^Q$ has dimension $2$.

Recall that $H_1(U) \simeq \langle y_0y_1 \rangle$.
We introduce some notation in order to filter $M$ by powers of $\langle y_0 y_1\rangle$.
Given $m \in M$, write $m = \sum_{0 \leq i,j \leq p-1} a_{i,j} y_0^i y_1^j$.
Let $[m]_k = \sum_{k={\rm min}\{i,j\}} a_{i,j} y_0^i y_1^j$.
For example, \[m_0=a_{0,0} + a_{1,0}y_0 + a_{0,1}y_1 + \cdots a_{p-1,0} y_0 + a_{0,p-1} y_1.\]
Then $m= \sum_{k = 0}^{p-1} [m]_i$ and $[m]_i \in \langle y_0 y_1\rangle^i - \langle y_0 y_1\rangle^{i+1}$.
The coset of $\psi(m)$ is represented by $[m]_0$.  
It suffices to show that ${\rm dim}(\{[m]_0 \mid m \in M^Q\})=2$. 

If $m \in M^Q$, then $(B_q-1)m=0$ for all $q \in Q$.  
This implies that $[(B_q-1)m]_1=0$.
Since $B_q - 1 \in \langle y_0 y_1 \rangle$, this implies that 
$[(B_q-1)[m]_0]_1=0$.  

We now isolate the term of lowest degree in $[m]_0$.
Let $\ell$ be minimal such that $a_{i,0}$ and $a_{0,j}$ are zero for all $i,j < \ell$.
By \Cref{cor:BqinI3}(1), $B_q-1 \equiv \alpha y_0y_1(y_0 + y_1)$ modulo $\langle y_0, y_1 \rangle^4$. In fact,
\[B_q-1 = y_0y_1 \sum_{h =1}^{p-2} b_h(y_0^h +y_1^h) \bmod \langle y_0y_1 \rangle^2 \]
for some coefficients $b_h$, where $b_1 \not = 0$ for at least one $q \in Q$ by
\Cref{cor:BqinI3}(2).
The condition $[(B_q-1)[m]_0]_1=0$ implies that 
\[ 
0  = [ [m]_0  \sum_{h= 1}^{p-2} b_h(y_0^h +y_1^h) ]_0 = \sum_{h =1}^{p-2} \sum_{q \geq \ell} b_h( a_{q,0} y_0^{h+q} + 
a_{0,q} y_1^{h+q} )  
\]
This shows that $\ell=p-1$ since $b_1 \not = 0$ and at least one of $a_{\ell,0}, a_{0,\ell}$ is non-zero.
\end{proof}

For $p \geq 5$, let $s_1 = y_0^{p-2} y_1^{p-2}$ and $a_1 = y_0^{p-3} y_1^{p-3}(y_0-y_1)$.

\begin{lemma}\label{lem:otherinvariants}
If $p \geq 5$, then $s_1, a_1 \in M^Q \cap H_1(U)$, so ${\rm dim}(M^Q) \geq 2p+1$
and ${\rm dim}(M^Q \cap H_1(U)) \geq 2p-1$.
\end{lemma}

\begin{proof}
By \cref{cor:BqinI3}, if $p \geq 5$, then $B_q - 1 \equiv \alpha y_0y_1(y_0 + y_1) \bmod \langle y_0,y_1\rangle^4$, 
for some constant $\alpha\in \Fp$. The given elements $s_1$ and $a_1$ annihilate the ideal $ \langle y_0,y_1\rangle^4$; moreover, 
\[s_1 y_0y_1(y_0 + y_1) = y_0^{p-1}y_1^{p-1}(y_0 + y_1) = 0,\] and likewise
\[ a_1 y_0y_1(y_0 + y_1) = y_0^{p-2}y_1^{p-2}(y_0^2 + y_1^2) = 0.\]
\end{proof}

Here is some data about $M^Q$ when $p=3,5,7$.

\begin{example}
\[\begin{array}{c|c|c}
p & {\rm dim}(M^Q) & {\rm dim}(M^Q \cap H_1(U)) \\
\hline
3 & 5 & 3\\
5 & 11 & 9 \\
7 & 17 & 15
\end{array}\]
\end{example}



\begin{example} \label{Exdata}
\begin{enumerate}
\item When $p=3$, then $M^Q = L= \Ker(B_{\tau_0} -1) \subset \Ker(B_{\tau_1}-1)$.

\item When $p=5$, then $M^Q={\rm Span}(L, s_1,a_1)$
As an ideal, $M^Q$ is generated by $\eta_0=y_0^4$, $\gamma_0 = y_1^4$, and $a_1$.
Also, $\Ker(B_{\tau_i}-1)$ is the same $13$-dimensional subspace for $1 \leq i \leq 4$.

\item When $p=7$, then the set $\{s_1,a_1, s_2, a_2\}$ extends a basis of $L$ to a basis of $M^Q$, where
\[s_2= y_0^{3} y_1^{3} (y_0^2-y_0y_1 + y_1^2) + y_0^4y_1^5,\] 
\[a_2 = y_0^2y_1^2(y_0^3-y_0^2y_1+y_0y_1^2-y_1^3) + y_0^3y_1^4(y_0-2y_1) - y_0^4y_1^5.\]
Also, $\Ker(B_{\tau_i}-1)$ is the same $19$-dimensional subspace for $1 \leq i \leq 6$.
\end{enumerate}
\end{example}

\begin{remark}
We would be able to say more about $M^Q$ for $p \geq 11$ if the following question has a positive answer.
\end{remark}

\begin{question}\label{QuestInvariant}
Is it true that 
$\Ker (B_{\tau_i}-1) = \Ker(B_{\tau_j}-1) $ for all $1\leq i,j \leq r$? 
If yes, this would imply that $M^Q={\Ker}(B_{\tau_0}-1) \cap {\Ker}(B_{\tau_1}-1)$.
By \Cref{Exdata}, the answer is yes when $p=3,5,7$.
\end{question}

\subsection{A comparison of invariant subspaces for different automorphisms} \label{SQinvariant2}

Let $B_{i}=B_{\tau_i}$ where $\tau_1, \ldots, \tau_r$ are the 
chosen generators of $Q$.  Note that $(B_{ia})^a = B_{\tau_{ia}^a}$.
Let $\rho_a \in {\rm Aut}(M)$ be given by the permutation action 
$\epsilon_0^i \epsilon_1^j \mapsto \epsilon_0^{ia} \epsilon_1^{ja}$.

The following result does not answer the first part of \cref{QuestInvariant}, but still gives a relation between the kernels of various $(B_i -1)$.

\begin{lemma}\label{Biaa=rhoBi}
 Let $a \in (\ZZ/p)^*$.  Then $(B_{ia})^a = \rho_a(B_i)$ for $i\neq 0$ and $B_0 = \rho_a (B_0)$.
\end{lemma}

\begin{proof}
By \cref{GalLQ_semidirect_product}, we may identify $a$ with an element of $\rGal{L}{\QQ}$. 
Then \[a \cdot(B_i \beta)  = a \cdot (\tau_i \cdot \beta) = (a \tau_i)\cdot\beta.\]

Consider $a \cdot (B_i \beta)$; recall that $B_i$ is an element of $\Lambda_1 = \Aa[\mu_p\times \mu_p]$, and the definition of the action of $\Lambda_1$ on $H_1(U,Y;\Aa)$ is via the map $\mu_p \times \mu_p \to {\rm Aut}(X)$ given by  $\e^i_0 \times \e^j_1: (x,y) \mapsto (\e^i_0 x, \e^j_1y).$ It follows that $a \cdot (B_i \beta) = \rho_a (B_i) (a \cdot \beta)$. 

On the other hand, note that $a \tau_i = (a \tau_i a^{-1}) a$. By \cref{GalLQ_semidirect_product}, we may identify $(a \tau_i a^{-1})$ with $(\tau_{ia})^a$ when $i \neq 0$, and with $\tau_0$ when $i=0$. 
Therefore, 
\[ \rho_a (B_i) (a \cdot \beta)= \begin{cases} (\tau_{ia})^a  \cdot (a \cdot \beta) = B_{ia}^a (a\cdot \beta)&\mbox{if } i \neq 0 \\ 
\tau_0 \cdot (a  \cdot\beta) = B_0  (a  \cdot\beta)  & \mbox{if } i =0 \end{cases}.\]
Because $H_1(U,Y;\Aa)$ is identified with the $\Lambda_1$-orbit of $\beta$, 
there exists an invertible $B'_a \in \Lambda_1$ such that $a \cdot \beta = B'_a \beta$. In the above identification, we can cancel this element and obtain
\[ \rho_a (B_i) = \begin{cases} (B_{ia})^a  &\mbox{if } i \neq 0 \\ 
B_0  & \mbox{if } i =0 \end{cases}.\]
\end{proof}

\begin{proposition} \label{Pkernel}
If $1 \leq i \leq r$ and $a \in (\ZZ/p)^*$, then ${\rm Ker}(\tau_{ai} -1) = \rho_a {\rm Ker}(\tau_i-1)$ is an equality of subsets of $H_1(U,Y;\Aa)$.
\end{proposition}

\begin{proof}
Since $((B_{ia})^a - 1) = ( B_{ai}^{a-1} \ldots +B_{ai}^2+B_{ai}+ 1)(B_{ai} - 1)$, it follows that 
\[ {\rm Ker}(B_{ai} - 1) \subseteq {\rm Ker}(B_{ai}^a - 1).\] 
By \cref{Biaa=rhoBi}, $ {\rm Ker}(B_{ai}^a - 1) = \rho_a  {\rm Ker}(B_i -1).$ Thus 
\[{\rm Ker}(B_{ai} - 1) \subseteq \rho_a {\rm Ker}(B_i -1),\] 
and it follows that 
\[{\rm Ker}(B_{i} - 1) \subseteq \rho_a {\rm Ker}(B_{a^{-1}i} -1).\]
Applying this equality repeatedly, we conclude 
\[{\rm Ker}(B_{i} - 1) \subseteq \rho_a {\rm Ker}(B_{a^{-1}i} -1) \subseteq  \rho_a^2 {\rm Ker}(B_{a^{-2}i} -1) \subseteq \ldots \subseteq (\rho_a)^j  {\rm Ker}(B_{a^{-j}i} -1)\] 
for any $j=1,2,\ldots$. Since $a^{p-1} = 1$ mod $p$, taking $j=p-1$ allows one to conclude that all of the inclusions are equalities.  Thus \[{\rm Ker}(B_{ai} - 1) = \rho_a {\rm Ker}(B_i -1).\]
\end{proof}

\section{Galois cohomology calculations}\label{SGalCohomology}

The goal of this section is to give a method for the efficient computation of the first cohomology group $H^1(G, M)$, where $M $ is the homology group $H_1(U,Y; \Aa)$, and $G$ is the Galois group of a suitable extension of $L$ over the cyclotomic field $K=\Q(\zeta)$. In future applications, the extension of $L$ will be its maximal extension ramified only over $p$, or various subextensions of it. As it is difficult to know explicitly the structure of such a group $G$ in general, the direct description of $H^1(G,M)$ in terms of crossed homomorphisms will not give an effective method for computation.

More generally, consider an extension of finite\footnote{everything in this section works for profinite groups and continuous cohomology as well} groups
\[
1 \to N \to G \to Q \to 1,
\]
and a $G$-module $M$. We are interested in determining the first cohomology group $H^1(G,M)$. The Lyndon-Hochschild-Serre spectral sequence gives rise to a long exact sequence
\[ 0 \to H^1(Q,M^N) \xrightarrow{inf} H^1(G,M) \xrightarrow{res} H^1(N, M)^Q \xrightarrow{d_2} H^2(Q,M^N) \to \dots \]
in which the differential $d_2$ can be identified with the transgression map \cite[2.4.3]{Neukirch-Schmidt-Wingberg}, and explicitly constructed as such.
Thus the computation of $H^1(G,M)$ reduces to a computation of $H^1(Q,M^N)$, the kernel of the transgression differential $d_2$, 
and the extension formed from those two.

We restrict our attention to the case when the normal subgroup $N$ acts trivially on the module $M$, since our intended application satisfies that assumption.

\subsection{The transgression}

To begin, note that the extension $G$ is determined by its factor set $\omega: Q \times Q \to N$ \cite[6.6.5]{Weibel}. Explicitly, let $s:Q \to G$ be an arbitrary set-theoretic section of the projection $G\to Q$, such that $s(1)=1$. Then the map
\begin{align}
\omega(q_1,q_2) = s(q_1)s(q_2)s(q_1q_2)^{-1},
\end{align}
is a cocycle, which is independent of the choice of section $s$ when viewed as an element of $H^2(Q,N)$ \cite[6.6.3]{Weibel} or \cite[IV.3]{Brown}.

The next proposition is similar to some material in \cite[Section 1]{sharifi}.

\begin{proposition}\label{prop:kerd2v1}
Let $G$ be an extension of $Q$ by $N$ determined by the factor set $\omega$, and let $M$ be a $G$-module on which $N$ acts trivially. Then the transgression
\[ d_2: H^1(N,M)^Q \to H^2(Q,M) \]
is given by
\[ d_2(\phi) = - \phi\circ \omega.\]
\end{proposition}

\begin{proof}
By \cite[2.4.3]{Neukirch-Schmidt-Wingberg}, the transgression in the Hochschild-Serre spectral sequence is given by \cite[1.6.6]{Neukirch-Schmidt-Wingberg}. By \cite[3.7 (3.9) and (3.10)]{Koch}, the map defined to be the transgression given in \cite[3.7]{Koch} coincides with the map given by \cite[1.6.6]{Neukirch-Schmidt-Wingberg}.

We may thus use the description of the transgression given in  \cite[3.7]{Koch}. Given $\phi: N\to M$ which represents an element in $H^1(N,M)^Q$, we construct an extension $\tilde \phi :G \to M$ as prescribed by \cite[3.7]{Koch}: Fix the same section $s:Q \to G$ as in 
the definition of the factor set $\omega$.
Since $N$ acts trivially on $M$, we can choose $\tilde \phi (s(q)) = 0$, for any $q \in Q$. Any element $g\in G$ can be written as $g = n s(q)$, with $n\in N, q\in Q$; for this $g$ we define $\tilde \phi(g) = \phi(n)$. The transgression $d_2 \phi :Q\times Q \to M $ is then given by
\[ d_2\phi (q_1,q_2) = \tilde \phi(s(q_1)) +s(q_1) \tilde\phi(s(q_2)) - \tilde \phi(s(q_1)s(q_2)) =  -  \tilde \phi(s(q_1)s(q_2)), \]
as the first two terms are both zero.
Now note that 
\[ s(q_1)s(q_2) = s(q_1)s(q_2) s(q_1q_2)^{-1} s(q_1q_2 ) = \omega(q_1,q_2) s(q_1q_2); \]
since $\omega(q_1,q_2) $ is in $N$, the definition of $\tilde \phi$ yields that
\[ d_2\phi (q_1,q_2) = - \tilde\phi (\omega(q_1,q_2) s(q_1q_2)) = - \phi(\omega(q_1,q_2)).\]
\end{proof}

\subsection{$H^*(Q,M)$, when $Q$ is elementary abelian}

It is well known that the cohomology group $H^1(Q,M)$ consists of crossed homomorphisms $Q\to M$ modulo the principal ones. This description can be seen as coming from the canonical bar resolution of the trivial module $\Z$. For our applications, however, it is also convenient to use the fact that $Q $ is assumed to be elementary abelian of rank $r+1$ (where $r=\frac{p-1}{2}$), i.e., $Q \cong C_p^{r+1}$, and use the resolution coming from tensoring $(r+1)$ minimal $C_p$-resolutions. We will use the resulting chain complex for computing $H^1(Q,M)$.  
More importantly, in the next subsections, we will use a comparison between cocycles of these different resolutions in order to obtain a more direct criterion equivalent to \cref{prop:kerd2v1} in \cref{thm:d2Explicit,Ckerneld2}. As we will delve pretty deeply into the inner workings of these resolutions, we start by recalling their constructions.

\subsubsection{The canonical or bar resolution}  \label{SresA}

For $i \geq 0$, let $B_i = \Z[Q^{i+1}] \cong \Z[Q]^{\otimes (i+1)}$.
Then $B_i \simeq \ZZ[Q] \otimes B_{i-1}$ for $i \geq 1$.
Thus, $B_i$ is a free $\ZZ[Q]$-module generated by elements of the form $[q_1\otimes \dots \otimes q_i]$, with each $q_i\in Q$. 
There is a free resolution
\begin{align}
B_\bullet = \{ \cdots \to B_2 \to B_1 \to B_0 \} \to \Z,
\end{align}
where the differential $d:B_n\to B_{n-1}$ is given by $d=\sum_{i=0}^n (-1)^id_i$, and each $d_i$ is the $\Z[Q]$-equivariant map determined by
\begin{align*}
d_0([g_1\otimes \cdots \otimes g_n]) &= g_1 \cdot [g_2\otimes \cdots \otimes g_n],\\
d_i([g_1\otimes \cdots \otimes g_n]) &= [g_1\otimes \cdots \otimes g_i g_{i+1} \otimes \cdots \otimes g_n], \  \text{for } 1\leq i \leq n-1,\\
d_n([g_1\otimes \cdots \otimes g_n]) &= [g_1\otimes \cdots \otimes g_{n-1}].
\end{align*}

In particular, $d:B_1 \to B_0$ is given by $d([g_1]) = g_1 \cdot [1] - [1]$ 
and $d:B_2 \to B_1$ is given by 
$d([g_1 \otimes g_2])= g_1 \cdot [g_2] - [g_1g_2] + [g_1]$.

\subsubsection{The tensor complex of minimal $C_p$-resolutions} \label{SresB} 

Let $\tau$ be a generator of $C_p$; then the complex
\[ C_\bullet = \{ \cdots \Z[C_p] \xrightarrow{1-\tau} \Z[C_p] \xrightarrow{N_{\tau}} \Z[C_p] \xrightarrow{1-\tau} \Z[C_p] \} \to \Z \]
is a free resolution of the trivial $\Z[C_p]$-module $\Z$. Now $\Z[Q] \cong \otimes_{j=0}^r \Z[C_p]$.
Thus a free resolution of the trivial $\Z[Q]$-module $\Z$ is given by the (totalization of the) tensor 
complex $\otimes_{j=0}^r C_\bullet$. 

To make this brutally explicit, 
for $0 \leq j \leq r$, let $C_{\bullet,j}$ denote the same complex as $C_\bullet$ but with the generator of $C_p$ denoted as $\tau_j$. 
For $i \geq 0$, the $i$th entry of the complex $C_{\bullet,j}$ is $C_{i,j} \cong \Z[C_p]$, and the map
$d_{i,j}:C_{i,j} \to C_{i-1,j}$ is multiplication by $\pm (1-\tau_j)$ if $i$ is odd and multiplication by $N_{\tau_j}$ if $i$ is even.

Therefore, $A_\bullet = {\rm Tot}(\otimes_{i=0}^r C_\bullet)$ has 
\[ A_n = \bigoplus_{i_0 +\cdots + i_{r} = n} C_{i_0, 0} \otimes \cdots \otimes C_{i_{r}, r} \cong \bigoplus_{i_0 +\cdots + i_{r} = n} \Z[Q].  \]

In particular, $A_0 \cong \Z[Q]$,
$A_1 \cong \Z[Q]^{r+1}$, and $A_2\cong \Z[Q]^{\rho}$,
where the exponent $\rho:=r+1+\binom{r+1}{2} = \frac{(p+1)(p+3)}{8}$ is the number of ways to partition $2$ into $r+1$ non-negative integers.

We need to define $A_1$ and $A_2$ more explicitly in order to describe the differential maps
$d:A_1 \to A_0$ and $d:A_2 \to A_1$.   
Since the notation is elaborate, first consider an example when $p=3$ and $r=1$.
Let $\sigma=\tau_0$ and $\tau=\tau_1$, then the complex is:
\[ \xymatrix@R=0pt@C+15pt{
	A_0 & A_1 & A_2 \\
	& 		& C_0 \otimes C_2 \ar[ld]_{N_\tau}\\
	& C_0 \otimes C_1  \ar[ld]^{1-\tau}
			& \oplus \\
C_0 \otimes C_0	& \oplus	& C_1 \otimes C_1 \ar[lu]_{-(1-\sigma)} \ar[ld]^{1-\tau} \\
	& C_1 \otimes C_0  \ar[lu]^{1-\sigma}		& \oplus \\
	&	 	& C_2 \otimes C_0  \ar[lu]^{N_\sigma}.
}\]
\begin{remark}\label{rem:signs}
Recall that negative signs must be introduced in the totalization of a double complex in order to make the differentials square to zero; see for example \cite[p.8]{Weibel}.
\end{remark}

More generally, 
recall that $A_n$ is a direct sum of submodules of the form 
\[S({\vec{v}})= C_{i_0, 0} \otimes \cdots \otimes C_{i_{r}, r} \cong \ZZ[Q],\]
where the entries of $\vec{v} = (i_0, \ldots, i_r)$ are non-negative numbers adding up to $n$.
For $n=1$, define $\vec{v}_j$ to have $j$th entry $1$ and all other entries $0$.
Then 
\[A_1= \bigoplus_{0 \leq j \leq r} S(\vec{v}_j).\]
For $n=2$, define $\vec{u}_j$ to have $j$th entry $2$ and all other entries $0$; and,
for $0 \leq j < k \leq r$, define $\vec{t}_{j,k}$ to have $j$th and $k$th entries $1$ and all other entries $0$.
Then 
\[A_2 = \left( \bigoplus_{0 \leq j \leq r} S(\vec{u}_j) \right) \oplus \left(\bigoplus_{0 \leq j < k \leq r} S(\vec{t}_{j,k}) \right).\]

The following results are now straightforward.

\begin{lemma} \label{dA1A0}
Writing $\alpha_1 \in A_1$ as $\alpha_1 = \oplus_{0 \leq j \leq r} g_j$ with $g_j \in S(\vec{v}_j)$, 
the differential $d:A_1 \to A_0$ is given by
\[ d(\alpha_1) = d(g_0, \dots, g_r) = \sum_{j=0}^r (1-\tau_j)g_j. \]
\end{lemma}

\begin{lemma} \label{dA2A1}
The differential $d:A_2 \to A_1$ is defined using the following maps on the given components
(and the zero map everywhere else)
\begin{eqnarray*}
d_2=N_{\tau_j} & : & S(\vec{u}_j) \to S(\vec{v}_j),\\
-d_1=-(1-\tau_j) & : & S(\vec{t}_{j,k}) \to S(\vec{v}_j),\\
d_1 = (1-\tau_k) & : & S(\vec{t}_{j,k}) \to S(\vec{v}_k).
\end{eqnarray*}

In other words, writing $\alpha_2 \in A_2$ as 
\[\alpha_2 = (\oplus_{0 \leq j \leq r} g_j,  \oplus_{0 \leq j < k \leq r} h_{j,k}),\] 
with $g_j \in S(\vec{u}_j)$ and $h_{j,k} \in S(\vec{t}_{j,k})$,
then $d(\alpha_2)= \oplus_{0 \leq j \leq r} \beta_j$ where 
\[\beta_j = N_{\tau_j} g_j - \sum_{k < j} (1- \tau_k) h_{k,j} + \sum_{k>j} (1-\tau_k) h_{j,k}.\]
\end{lemma}

Again, the negative signs in front of some of the $d_1$'s are because of \cref{rem:signs}.

\begin{remark}\label{SH1Q}
Using \cref{dA1A0,dA2A1}, it is possible to compute $H^1(Q,M)$ directly.
For example, for small $p$, we used Magma to explicitly calculate $H^1(Q,M)$; here is a table for its dimension:
\[\begin{array}{c|c}
p & {\dim}(H^1(Q, M)) \\
\hline
3 & 9\\
5 & 33\\
7 & 68
\end{array}.\]
More information about the relationships between the kernels and images of $B_i-1$ as $i$ varies, as in \cref{QuestInvariant}, may yield a result for general $p$ along these lines. 
\end{remark}

\subsubsection{Comparison of resolutions}
The resolutions $A_\bullet$ and $B_\bullet$ constructed above are both injective resolutions of the trivial $Q$-module $\Z$. Therefore, by abstract nonsense, there is a quasi-isomorphism $f_\bullet : A_\bullet \to B_\bullet$, with each $f_i:A_i \to B_i$ being $Q$-equivariant. The goal of this subsection is to construct $f_0,f_1,f_2$. In fact, we will take $f_0$ to be the identity map on $A_0\cong B_0 = \Z[Q]$. The next two results determine $f_1$ and $f_2$ explicitly.

\begin{lemma} \label{dA1B1}
Write $\alpha_1 \in A_1$ as $\alpha_1 = \oplus_{0 \leq j \leq r} g_j$ with $g_j \in S(\vec{v}_j)$. 
Define $f_1: A_1 \to B_1$ by
\[f_1(\alpha_1)= f_1 (g_0,\dots, g_r) = - \sum_{j=0}^{r} g_j[\tau_j].\]
Then the following diagram commutes
\[ \xymatrix{ 
A_1 \ar[r]^{d^A}\ar[d]_{f_1} & A_0 \ar[d]_{\rm id} \\
B_1 \ar[r]^{d^B} & B_0.
}\]
\end{lemma}

\begin{proof}
Let $1_j \in S(\vec{v}_j) \subset A_1$ be 
the element such that $g_j=1$ and all other coordinates are zero.
By Lemma \ref{dA1A0}, ${\rm id}(d^A(e_j)) = 1-\tau_j$.
By definition $f_1(e_j)=-[\tau_j]$, which equals $d^B(f_1(e_j))=-(\tau_j-1)$.
Since $\{e_j\}$ generate $A_1$ as a $\Z[Q]$-module and all the maps are $Q$-equivariant,
the diagram commutes in general.
\end{proof}

\begin{lemma} \label{dA2B2}
Write $\alpha_2 \in A_2$ as 
$\alpha_2 = (\oplus_{0 \leq j \leq r} g_j,  \oplus_{0 \leq j < k \leq r} h_{j,k})$,
with $g_j \in S(\vec{u}_j)$ and $h_{j,k} \in S(\vec{t}_{j,k})$.
Define $f_2:A_2 \to B_2$ as follows:
\[f_2(\alpha_2) = - \sum_{j=0}^r g_i [N_{\tau_i} \otimes \tau_i] + 
\sum_{0 \leq j < k \leq r} h_{j,k} ( \tau_k \otimes \tau_j - \tau_j \otimes \tau_k).\]
Then the following diagram commutes
\[ \xymatrix{ 
A_2 \ar[r]^{d^A}\ar[d]_{f_2} & A_1 \ar[d]^{f_1} \\
B_2 \ar[r]^{d^B} & B_1.
} \]
\end{lemma}

\begin{proof}
By Lemma \ref{dA2A1},
$d^A(\alpha_2)= \oplus_{0 \leq j \leq r} \beta_j$ where 
\[\beta_j = N_{\tau_j} g_j - \sum_{k < j} (1- \tau_k) h_{k,j} + \sum_{k>j} (1-\tau_k) h_{j, k}.\] 
Let $1_j \in S(\vec{u}_j) \subset A_2$ be the element such that $g_j=1$ and all other coordinates are zero.
Then $f_1(d^A(1_j)) = - N_{\tau_j} [\tau_j]$.
By definition, $f_2(1_j)= -[N_{\tau_j} \otimes \tau_j]$.   Since $N_{\tau_j} \tau_j=N_{\tau_j}$, 
it follows that
\[d^B(f_2(1_j)) = -( N_{\tau_j}[\tau_j]-[N_{\tau_j} \tau_j ]+[N_{\tau_j}]) = - N_{\tau_j} [\tau_j].\]

Finally, let $1_{j,k} \in S(\vec{t}_{j,k}) \subset A_2$ be the element such that $h_{j,k}=1$ and all other coordinates are zero.
Then
\[d^A(1_{j,k})= (1-\tau_k) e_j -(1-\tau_j)e_k,\] and 
\[ f_1(d^A(1_{j,k})) =  f_1((1-\tau_k) e_j -(1-\tau_j)e_k) = -(1-\tau_k)[\tau_j] + (1-\tau_j)[\tau_k].\]
By definition, $f_2(1_{j,k})=\tau_k \otimes \tau_j - \tau_j \otimes \tau_k$.
Then
\begin{eqnarray*}
d^B([\tau_k \otimes \tau_j] - [\tau_j \otimes \tau_k]) & = &
(\tau_k[\tau_j] - [\tau_k \tau_j]+[\tau_k] )- (\tau_j[\tau_k] - [\tau_j \tau_k] + [\tau_j]) \\
& = & (\tau_k-1)[\tau_j] - (\tau_j-1)[\tau_k].
\end{eqnarray*}
Since $\{1_j, 1_{j,k}\}$ generate $A_2$ as a $\Z[Q]$-module and all the maps are $Q$-equivariant,
the diagram commutes in general.
\end{proof}

\subsection{Comparison of cocycles}

In the above, we constructed two resolutions of the trivial $Q$-module $\Z$, and explicitly constructed a map between them in low degrees. Now we investigate what this tells us in cohomology. Namely, we know that 
\[ H^*(Q,M) = \Ext^*_{\Z[Q]}(\Z,M), \]
and the latter can be computed as either $H^*\Hom_{\Z[Q]}(A_\bullet, M )$ or $H^*\Hom_{\Z[Q]}(B_\bullet, M )$. The map $f_\bullet$ gives us a way to compare these two approaches.

Consider a $1$-cocycle $a \in H^1(Q,M)$. 
Let $\phi: Q \to M$ be a bar resolution representative of $a$, so that the class of $\phi$ in $H^1(Q,M)$ is $a$.
Then $\phi$ can be uniquely extended to (and encodes the information of) a $\Z[Q]$-module map $\tilde \phi: \Z[Q]^{\otimes 2} \to M$. 
A representative of $a$ in the $A_\bullet$ resolution is the composition $\psi = \tilde{\phi} \circ f_1$, namely
\[ \psi: A_1 \cong \Z[Q]^{r+1} \xrightarrow{f_1} B_1 \cong \Z[Q]^{\otimes 2} \xrightarrow{\tilde \phi} M. \] 
Now $\psi$ is a $\Z[Q]$-equivariant map determined by its values on the generators $e_j$ of $A_1$.
By \cref{dA1B1},
\[ m_j:= \psi(e_j) = \tilde\phi(-[\tau_j]) = - \phi(\tau_j),\]
giving the following result.

\begin{lemma} \label{Lcomparecycle1}
In the resolution $\Hom_{\Z[Q]}(A_\bullet, M )$, which starts as
\[ M \to M^{r+1} \to M^{\rho} \to \cdots, \]
the tuple $(m_0, \ldots, m_r)=(-\phi(\tau_0), \ldots, -\phi(\tau_r)) \in M^{r+1}$  represents
the class $a \in H^1(Q,M)$ of the map $\phi:Q\to M$. 
\end{lemma}

Next, consider a $2$-cocycle $b \in H^2(Q,M)$.
Let $\varphi:Q\times Q \to M$ represent $b$. 
The map $\varphi$ uniquely determines a $\ZZ[Q]$-equivariant map 
$\tilde \varphi: B_2\cong \Z[Q]^{\otimes 3} \to M$, by extending $\Z[Q]$-linearly.
A representative of $b$ in the $A_\bullet$ resolution is the composition
\[ \theta: A_2 \cong \Z[Q]^{\rho} \xrightarrow{f_2} B_2  \xrightarrow{\tilde \varphi} M. \]
The map $\theta$ is determined by its values on the $\Z[Q]$-generators $1_j$ and $1_{j,k}$ of $A_2$.
By \cref{dA2B2},
\begin{align*}
n_j:=\theta(1_j) &= \tilde\varphi([ -N_{\tau_j} \otimes \tau_j]) 
= - \tilde \varphi(N_{\tau_j}, \tau_j) = - \sum_{i=0}^{p-1} \varphi(\tau_j^i,\tau_j) ,\\
n_{j,k}:=\theta(1_{j,k}) &= \tilde\varphi([\tau_k \otimes \tau_j] - [\tau_j \otimes \tau_k]) 
= \varphi(\tau_k,\tau_j)-\varphi(\tau_j,\tau_k),\\
\end{align*}
proving the following result.

\begin{lemma} \label{Lcomparecycle2}
In the resolution $\Hom_{\Z[Q]}(A_\bullet, M )$, which starts as
\[ M \to M^{r+1} \to M^{\rho } \to \cdots, \]
the tuple $(n_j, n_{j,k}) \in M^{\rho}$ defined above represents the class $b \in H^2(Q,M)$ 
of the map $\varphi:Q\times Q \to M$. 
\end{lemma}

\subsection{The kernel of $d_2$, revisited}

Using the comparison of cocycles from the previous section, we give a more direct description of the kernel of the transgression $d_2: H^1(N,M)^Q \to H^2(Q,M^N)$ (compared to what \cref{prop:kerd2v1} implies), when $N$ acts trivially on $M$ and $Q$ is elementary abelian.

We set up some notation associated to the extension 
\begin{equation} \label{shortexact}
1\to N \to G \to Q \to 1.
\end{equation}
We assume that $Q$ is elementary abelian of rank $(r+1)$; choose generators of $Q$ and denote them by $\tau_i$, with $0\leq i \leq r$. To define the factor set $\omega$, we used a section $s:Q \to G$ (and noted that as a cohomology element, $\omega$ does not depend on $s$). Without loss of generality, we can assume not only that $s(1)=1$, but also 
\[ s(\tau_0^{t_0} \cdots \tau_r^{t_r}) = s(\tau_0)^{t_0}\cdots s(\tau_r)^{t_r}, \text{   for } 0\leq t_i \leq p-1.\]
For $0\leq j \leq r$, define elements $a_j \in N$ by
\[ a_j = s(\tau_j)^p,\]
and for $0 \leq j < k \leq r$, define $c_{j,k} \in N$ by
\[ c_{j,k} = [ s(\tau_k), s(\tau_j) ] = s(\tau_k) s(\tau_j) s(\tau_k)^{-1} s(\tau_j)^{-1} . \]
Recall that $\omega:Q\times Q \to N$ was defined as
\[ \omega(q_1,q_2) = s(q_1) s(q_2) s(q_1q_2)^{-1}. \]
Elementary calculation then yields the following result.
\begin{lemma}\label{Lomegaac}
If $0\leq j \leq r$ and $0 \leq t < p-1$, then $\omega(\tau_j^t, \tau_j)=0$ and $a_j = \omega(\tau_j^{p-1}, \tau_j)$.
If $0\leq j < k \leq r$, then $c_{j,k} = \omega(\tau_k,\tau_j)\omega(\tau_j, \tau_k)^{-1}$. 
\end{lemma}

\begin{theorem}\label{thm:d2Explicit}
The class of $\phi:N\to M$ is in the kernel of $d_2$ if and only if the 
tuple $(-\phi(a_j), \phi(c_{j,k})) \in M^{\rho}$ is in the image of the differential in $ \Hom_{\Z[Q]}(A_\bullet, M)$,
\[ d^M: M^{r+1} \to M^{\rho}\]
which, by Lemma \ref{dA2A1}, is explicitly given by
\[ d^M(m_0, \ldots, m_r) = (N_{\tau_j} m_j, - (1-\tau_j)m_k + (1-\tau_k)m_j).\] 
\end{theorem}

\begin{proof}
Consider a class in $H^1(N,M)^Q$ represented by a map $\phi:N\to M$.
By \cref{prop:kerd2v1}, $\phi \in {\rm Ker}(d_2)$ if and only if $\phi \circ \omega :Q\times Q \to M$ represents the zero class in $H^2(Q,M)$. (Note that this is the same as requiring that $-\phi \circ \omega$ represents zero.) 
This representative is given in the bar resolution, and we now translate the condition on $\phi \circ \omega$ to the $A_\bullet$-resolution as above. 

To find a representative for $\phi\circ \omega$ in the $A_\bullet $-resolution, we 
first extend $\omega$ to a $Q$-equivariant map $\tilde \omega: \Z[Q^{3}] \to N$ 
and then take the composition $\tilde{\omega} \circ f_2$.
By \cref{dA2B2,Lcomparecycle2}, $\phi\circ \omega$ is represented by the tuple 
$(n_j^\phi, n_{j,k}^\phi) \in M^{\rho}$, where
\[n_j^\phi = \phi(\tilde \omega ( f_2(1_j))) = \phi(\tilde \omega( -N_{\tau_j} \otimes \tau_j ) ) 
= -\sum_{i=0}^{p-1} \phi(\omega(\tau_j^i, \tau_j)).\]
By \cref{Lomegaac},
\begin{align*}
n_j^\phi &= -\phi(\omega(\tau_j^{p-1}, \tau_j)) = -\phi(a_j), \text{ and}\\
n_{j,k}^\phi &= \phi(\tilde \omega ( f_2(1_{j,k}))) 
= \phi(\tilde \omega ( [\tau_k \otimes \tau_j ] - [\tau_j\otimes \tau_k])  )= \phi(c_{j,k}).
\end{align*}
Applying \cref{Lcomparecycle1} now completes the proof.
\end{proof}

We return now to the situation of the Fermat curve.  

\begin{corollary} \label{Ckerneld2}
Suppose that $E/K$ is a finite Galois extension dominating $L/K$.
In the extension \eqref{shortexact}, let $Q={\rm Gal}(L/K)$ and $G = {\rm Gal}(E/K)$ and $N={\rm Gal}(E/L)$.
Recall that $N$ acts trivially on the relative homology $M=H_1(U, Y; A)$.

Assume $p\geq 5$. 
Then $\phi: N \to M$ represents an element in the kernel of $d_2$ if and only if for all $0\leq j \leq r$,
\[ \phi(a_j) = 0, \]
and there is an (r+1)-tuple $(m_0, \dots m_r) \in M^{r+1}$, such that
\[ \phi(c_{j,k}) = -(1-\tau_j)m_k +(1-\tau_k)m_j. \]
\end{corollary}

\begin{proof}
This follows from \cref{thm:d2Explicit}, since $N_{\tau_i}$ acts as zero on $M$ by \Cref{Tnorm}.
\end{proof}

\begin{remark}
We have a second, more direct proof of \Cref{thm:d2Explicit} as well.  The converse direction is long, computational, and rather unenlightening, hence we decided not to include it. Yet we sketch the forward direction 
here.
Note that $-\phi \in {\rm Ker}(d_2)$ 
if and only if the map $\phi \circ \omega: Q \times Q \to M$ represents the zero cohomology class in $H^2(Q,M)$; equivalently, 
$\phi \circ \omega$ is of the form 
\begin{equation} \label{coboundary}
dm:(q_1,q_2) \mapsto q_1 m(q_2) - m(q_1q_2) + m(q_1),
\end{equation} for some map $m: Q \to M$.  
Let $m_i=m(\tau_i)$.

If $dm = \phi \circ \omega$, then the values $m_j=m(\tau_j) \in M$ determine $m(q)$ for all $q \in Q$ 
because of the $Q$-action.
Specifically, by induction, one can show $m(\tau_j^{t+1})=(\sum_{\ell=0}^{t} \tau_j^\ell) \cdot m_j$ for $1 \leq t \leq p-2$.
Then $\phi \circ \omega(\tau_j, \tau_j^{p-1}) = \phi(a_j)$.
If $\phi \circ \omega = dm$, then $\phi(a_j)=\tau_j \cdot m(\tau_j^{p-1}) + m(\tau_j)$.
Thus $-\phi(a_j)=-N_{\tau_j} \cdot m_j$. 

Next, if $j< k$, then $m(\tau_j \tau_k) = \tau_j \cdot m_k+ m_j$, because $dm(\tau_j,\tau_k)=\omega(\tau_j, \tau_k)= 0$.
Recall that $\phi \circ \omega(\tau_k, \tau_j)=\phi(c_{j,k})$.
If $\phi \circ \omega = dm$, then
$\phi(c_{j,k})=\tau_k \cdot m_j - m(\tau_j \tau_k) + m_k$,
which simplifies to $-\phi(c_{j,k})=(1-\tau_k) \cdot m_j- (1-\tau_j) \cdot m_k$ by substitution.
\end{remark}

\section{Compatibility with points over finite fields} \label{finitefield}

In this final section, we study the action of Frobenius on schemes defined over a finite field of cardinality $\ell$.
In \cref{S7.1}, we use motivic homotopy theory to provide congruence conditions on the characteristic polynomials 
of Frobenius on mod $p$ cohomology.
In \cref{subsection:point_counts_finite_fields_Fermat_example}, we use this and information about 
$B_q$ to compute the $L$-polynomial of the degree $p$ Fermat curve modulo $p$.
The results in this section are not new, but they highlight 
important concepts emerging in the interaction between topology and number theory.

\subsection{Number of points modulo $p$} \label{S7.1}

Let $X$ be a smooth, proper scheme over $\F_{\ell}$.  Let $F$ denote the Frobenius morphism. 
Let $p$ be a prime number not dividing $\ell$. 

Let $N_m$ denote the number of points of $X$ defined over $\F_{\ell^m}$ for $m \in \NN$, 
and let $\overline{N}_m$ denote the reduction of $N_m$ mod $p$. 
By the Lefschetz trace formula, the values $N_m$ are determined by the action of $F$ on $H^*(X_{\overline{\F}_{\ell}}, \Q_p)$
and the values $\overline{N}_m$ are determined by the action of $F$ on $H^*(X_{\overline{\F}_{\ell}}, \F_p)$. 
This section contains a new proof of this fact for $\overline{N}_m$ using realization functors
which is made possible by the work of Hoyois \cite{Hoyois}.

Define $P_i(t)$ in $\mathbb{Q}_p[t]$ and $\overline{P}_i(t)$ in $\F_p[t]$ by 
\[P_i(t) = \det (1-Ft| H^i(X_{\overline{\F}_{\ell}}, \Q_p)), \ 
\overline{P}_i(t) =  \det (1-Ft| H^i(X_{\overline{\F}_{\ell}}, \F_p)).\] 
Define $Z(t)$ in $\Q_p[[t]]$ and $\overline{Z}(t)$ in $\FF_p[[t]]$ by 
\[Z(t) = \prod_{i=0}^{\infty} P_i(t)^{(-1)^{i+1}}, \ 
\overline{Z}(t) = \prod_{i=0}^{\infty} \overline{P}_i(t)^{(-1)^{i+1}}.\]
If $Q \in \F_p[[t]]$ is invertible (e.g., if $Q(0)=1$), let $\frac{d}{dt} \log Q = \frac{d}{dt}Q/Q$.

In this section, we prove the following result using motivic homotopy theory.

\begin{proposition}\label{modpWeil}
The mod $p$ number of points $\overline{N}_m$ of $X$ over $\F_{\ell^m}$ is determined by $\Sigma_{m=1}^{\infty} \overline{N}_m t^{m-1} = \frac{d}{dt} \log \overline{Z}(t).$
\end{proposition}

\Cref{modpWeil} follows from \cite[Section 3 Fonctions $L$ Modulo $\ell^n$ et Modulo $p$, Theorem 2.2 (b)]{sga4andhalf}. Here is a proof using motivic homotopy theory.

\begin{proof}
Let $\Tr$ denote the trace of an endomorphism of a strongly dualizable object in a symmetric monoidal category. The Frobenius $F$ is an endomorphism of $X$ viewed as an object the stable $\BA^1$-homotopy category of $\PP$-Spectra over $\F_{\ell}$. As $X$ is strongly dualizable, we have that $\Tr(F^m)$ lives in the Grothendieck-Witt ring $\GW(\F_{\ell})$. By Hoyois's generalized Lefschetz trace formula \cite[Example 1.6, Theorem 1.3]{Hoyois}, $\Tr(F^m) = N_m$. Applying the symmetric monoidal functor $H^*((-)_{\overline{\F}_{\ell}},\F_p)$, the trace $\Tr(F^m)$ becomes the trace in the symmetric monoidal category of graded $\F_p$ vector spaces, which is $\Sigma_i (-1)^i \Tr F^m|H^i(X_{\overline{\F}_{\ell}}, \F_p)$. Applying the same functor to the endomorphism $N_m$ of the sphere yields $\overline{N}_m$ regarded as an endomorphism of $\F_p$ viewed as a graded vector space concentrated in degree $0$. It follows that 
\begin{equation}\label{modpLT} \overline{N}_m = \Sigma_i (-1)^i \Tr F^m|H^i(X_{\overline{\F}_{\ell}}, \F_p).\end{equation}

The claimed equality then follows from a formal algebraic manipulation. One could apply \cite[Rapport sur la formula des traces 3.3.1]{sga4andhalf}, or to be explicit, proceed as follows.

Since $\overline{P}_i(0) = 1$, it follows that $\overline{P}_i(t) = \prod (1 - a_{i,j} t)$ for some $a_{i,j}$ in $\overline{\F}_p$. Since the matrix corresponding to the action of $F$ on $H^i(X_{\overline{\F}_{\ell}}, \F_p)$ can be put in upper triangular form over $\overline{\F}_p$, it follows that the diagonal entries are the $a_{i,j}$. Thus $\Tr F^m = \Sigma a_{i,j}^m$ for all $m$.

Furthermore, $\overline{P}_i$ is invertible in $\F_p[[t]]$ since $\overline{P}_i(0) =1$.  
Thus
\[\frac{d}{dt} \log \overline{P}(t) = \frac{\frac{d}{dt}\overline{P}(t)}{\overline{P}(t)}= 
-\sum_j \frac{a_{i,j}}{1 - a_{i,j} t} = - \sum_j \sum_m a_{i,j}^m t^{m-1}.\]
Also, 
\[\frac{d}{dt} \log \overline{Z}(t) = \frac{\frac{d}{dt}\overline{Z}(t)}{\overline{Z}(t)}.\]
Since $d/dt \log$ is a homomorphism,
\begin{align*}
\frac{d}{dt} \log \overline{Z}(t) & = - \sum_i (-1)^{i+1}  \sum_j \sum_m a_{i,j}^m t^{m-1} 
= \sum_i (-1)^i \sum_m \Big( \sum_j a_{i,j}^m\Big) t^{m-1}  \\
& = \sum_i  \sum_m (-1)^i \Big( \Tr F^m | H^i(X_{\overline{F}_{\ell}}, \F_p) \Big) t^{m-1} \\
&= \sum_m \overline{N}_m t^{m-1},
\end{align*}
where the last equality follows from \eqref{modpLT}. 
\end{proof}

\subsection{Application to the Fermat curve}\label{subsection:point_counts_finite_fields_Fermat_example}

Let $X$ be the Fermat curve of exponent $p$ over a prime $\ell$ of $\ZZ[\zeta_p]$. Let $\FF$ be the residue field of $\ell$, and $\FF_{\ell^{m}}$ denote the unique degree $m$ extension. 
Knowledge of $B_{\sigma}$ for $\sigma \in Q= \rGal{L}{K}$ and \Cref{modpWeil}
determine the zeta function of $X$ modulo $p$ as follows. 

\begin{proposition} \label{Lmodpzeta}
Let $X$ and $\FF$ be as above, and let $\Jac X$ denote the Jacobian of $X$.
\begin{enumerate}
\item $Z(X/\FF, T) \equiv (1-T)^{2g-2} \bmod p$.
If $N_m:=\#X(\FF_{\ell^{m}})$, then $N_m \equiv 0 \bmod p$ for all $m \geq 1$.
\item $Z(\Jac X/\FF, T) \equiv 1 \bmod p$.
If $N_m:=\#\Jac X(\FF_{\ell^{m}})$, then $N_m \equiv 0 \bmod p$ for all $m \geq 1$.
\end{enumerate}
\end{proposition}

\begin{proof}
Note that $Z(Y/\FF, T) \equiv \overline{Z}(Y/\FF,T) \bmod p$ for $Y=X$ or $\Jac X$. 
\begin{enumerate}
\item The action of the Frobenius $F$ on $M=H_1(U,Y;\F_p)$ is given by multiplication by $B_{\sigma}$, where $\sigma \in Q$ is the Frobenius for $\ell$. 
Now $H_1(X,\F_p)$ is a sub-quotient of $M$, and $M$ has a basis (namely the nilpotent basis given by monomials in $y_i=\e_i-1$) in which the action of $B_{\sigma}$ is lower-triangular with diagonal entries equal to $1$. Since $H^1(X,\F_p)$ is the linear dual of $H_1(X,\F_p)$, so it follows that the action of $F$ on $H^1(X,\F_p)$ satisfies $\det(1-F T | H^1(X,\F_p))  = (1-T)^{2g}$, proving the first claim.
For the second claim, note that 
\[Z(X/\FF_q, T)\equiv \frac{(1-T)^{2g}}{(1-T)(1-\vert \FF \vert T)} \equiv (1-T)^{2g-2} \bmod p,\] where the last equivalence follows because $\FF$ has a $p$th root of unity, implying $\vert \FF \vert -1 \equiv 0 \bmod p$.
By \Cref{modpWeil}, 
\[\Sigma_{m=1}^{\infty} \overline{N}_m T^{m-1} = d/dT \log \overline{Z}(T) = -(2g-2)(1-T)^{2g-3}/ \overline{Z}(T).\]
But $g = (p-1)(p-2)/2$, so $2g-2 = p^2 - 3p \equiv 0 \bmod p$.  
\item We have seen that the action of $F$ on $H^1(X,\F_p)$ is such that $1-F$ is nilpotent. Thus the same is true for the action of $F$ on the $i$th wedge power $\wedge^i H^1(X,\F_p)$. Since $H^i(\Jac X, \F_p) \cong \wedge^i H^1(X,\F_p)$, it follows that $\det(1-F T | H^i(\Jac X,\F_p))  = (1-T)^{d_i}$, where $d_i = {2g \choose i}$ is the dimension of  $\wedge^i H^1(X,\F_p)$. Thus \[Z({\rm Jac} X/\FF_q, T)\equiv (1-T)^{\sum_i (-1)^{i+1}d_i} \equiv 1 \bmod p. \] 
\end{enumerate}
\end{proof}

\begin{remark}
The facts in \cref{Lmodpzeta} can also be proven directly.
The fact that $N_m \equiv 0 \bmod p$ is a direct consequence of the fact that the $C_p \times C_p$ action on $X$ has 3 orbits of size $p$ and 
all other orbits of size $p^2$.

For the fact about the $L$-polynomial, let $\chi$ be a character of $\FF$ of order $p$.  
Let $J_{i,j}=J(\chi^i, \chi^j)=\sum_{a+b=1} \chi^i(a) \chi^j(b)$.
By \cite[page 98]{irelandrosen}, $\#X(\FF)=L^f+1 + \sum_S J_{i,j}$ where
$S=\{(i,j) \mid 1 \leq i,j \leq p-1, \ i + j \not \equiv 0 \bmod p\}$.  
Note that there are $2g=(p-1)(p-2)$ such pairs. 
In fact, by \cite[page 61]{katz}, the eigenvalue of Frobenius on the eigenspace of $H^1(X)$ corresponding to 
$(\chi^i, \chi^j)$ is $-J_{i,j}$.
Lemma \ref{Lmodpzeta} can also be proven using congruence properties of Jacobi sums and the fact that \[L(X/\FF, T) = \prod_S (1-J_{i,j}T).\]
\end{remark}

\bibliographystyle{amsalpha}
\bibliography{biblio}

\end{document}